\documentclass[12pt]{amsart}

\usepackage{latexsym,color,amsmath,amsthm,amssymb,amscd,amsfonts}
\usepackage{tikz}
\usetikzlibrary{arrows}
\usepackage{scalefnt}
\usepackage[font=small,labelfont=bf]{caption}
\usepackage{float}
\usepackage{graphicx}
\usepackage{subcaption}
\usepackage{relsize}
\usepackage{amsopn}
\usepackage{mathrsfs}  %
\usepackage[hidelinks]{hyperref}
\usepackage{bm}
\usepackage{bbm}
\usepackage[shortlabels]{enumitem}
\usepackage{multicol}
\usepackage{float} 
\usepackage{todonotes}
\usepackage{cite}
\usepackage{soul}
\usepackage{tabto}
\usepackage{hyperref}

\usepackage[margin=2.5cm]{geometry}

\graphicspath{ {images/} }
\newtheoremstyle{nonum}{}{}{\itshape}{}{\bfseries}{.}{ }{\thmnote{#3}}

\newtheorem{thm}{Theorem}[section]
\newtheorem*{thm*}{Theorem}

\newtheorem{cor}[thm]{Corollary}
\newtheorem{lem}[thm]{Lemma}
\newtheorem{prop}[thm]{Proposition}

\newtheorem{definition}[thm]{Definition}
\newtheorem*{definition*}{Definition}
\newtheorem{fact}[thm]{Fact}

\newcounter{examples}
\newtheorem{exm}[examples]{Example}

\newtheorem{manualtheoreminner}{Theorem}
\newenvironment{manualtheorem}[1]{%
	\renewcommand\themanualtheoreminner{#1}%
	\manualtheoreminner
}{\endmanualtheoreminner}

\newtheorem{manualdefinitioninner}{Definition}
\newenvironment{manualdefinition}[1]{%
	\renewcommand\themanualdefinitioninner{#1}%
	\manualdefinitioninner
}{\endmanualtheoreminner}

\theoremstyle{nonum}

\theoremstyle{remark}
\newtheorem*{rems*}{Remarks}

\newtheorem{rem}[thm]{Remark}

\newcommand{\R}{\mathbb R}
\newcommand{\RR}{\mathbb R}

\def\K{{\mathcal K}}
\def\A{{\mathcal A}}
\def\L{{\mathcal L}}
\def\P{{\mathcal P}}

\newcommand{\dual}{^\circ}

\def\orqi{order reversing quasi involution }
\def\orqis{order reversing quasi involutions}
\def\BS{Blaschke-Santal\'{o}~}
\def\PL{Pr\'ekopa-Leindler~}

\def\sp{{\rm sp }}

\newcommand{\iprod}[2]{\langle #1,#2 \rangle} %

\def\C{\mathcal C}

\newcommand{\vol}{{\rm{Vol}}}

\def\epi{{\rm epi}\,}
\def\hypo{{\rm hypo}\,}

\def\dom{{\rm dom}\,}
\def\cvx{{\rm Cvx}}

\def\eps{{\varepsilon}}
\def\dom{{\rm dom}}
\def\conv{{\rm conv}}

\def\A{{\mathcal A}}

\def\R{\mathbb R}

\def\cvx{{\rm Cvx}}

\def\L{{\mathcal L}}
\def\S{{\mathcal S}}

\def \sp#1{\langle #1\rangle}

\def\l{\lambda}
\def\v{\varphi}

\def\implies{\;\;\Rightarrow\;\;}

\def\dom{\text{dom}}

\makeatletter
\def\moverlay{\mathpalette\mov@rlay}
\def\mov@rlay#1#2{\leavevmode\vtop{%
		\baselineskip\z@skip \lineskiplimit-\maxdimen
		\ialign{\hfil$\m@th#1##$\hfil\cr#2\crcr}}}
\newcommand{\charfusion}[3][\mathord]{
	#1{\ifx#1\mathop\vphantom{#2}\fi
		\mathpalette\mov@rlay{#2\cr#3}
	}
	\ifx#1\mathop\expandafter\displaylimits\fi}
\makeatother

\begin{document}
\title{A zoo of dualities}

\author{S. Artstein-Avidan, S. Sadovsky, K. Wyczesany } %

	\begin{abstract}
		In this note we study \orqis~ and their properties. These maps are dualities (order reversing involutions) on their image. We prove that any \orqi is induced by a cost. Invariant sets of \orqis~ are of special interest and we provide several results regarding their existence and uniqueness. 
		We determine when an \orqi on a sub-class can be extended to the whole space and discuss the uniqueness of such an extension. We also provide several ways for constructing new \orqis~from given ones. In particular, we define the dual of an order reversing quasi involution. Finally, throughout the paper we exhibit a ``zoo'' of illustrative examples. Some of them are classical, some have recently attracted attention of the convexity community and some are new. We study in depth the new example of dual polarity and obtain a Blaschke-Santal\'o type inequality for a corresponding Gaussian volume product.
		The unified point of view on \orqis~ presented in this paper gives a deeper understanding of the underlying principles and structures, offering a new and exciting perspective on the topic, exposing many new research directions.  
	\end{abstract}

\maketitle

\section{Definitions and results} 

Duality theory in mathematics at large, and in geometry in particular, is a pervasive and important concept with manifestations in many areas. It has been a cornerstone in convexity theory, in which several characterization theorems have been proven (see, among others,  \cite{gruber1991endomorphisms, boroczky-schneider,slomka2011duality, florentin2012stability,segal2013duality,schneider2008endomorphisms, before,before-hidden, ArtsteinMilmanHidden,  iusem2015order,artstein-slomka2012order, artstein-florentin2012order}).  As we shall demonstrate, dualities show up in unexpected places,  fueling new research directions and affecting old ones.

In the context of this paper, a duality is defined as a mapping $T$ from a partially ordered set   $\C$   to itself that is order reversing and which satisfies $T\circ T = Id$ (where $Id$ is the identity operator).  In most cases $\C$ is a class of sets partially ordered by inclusion, that is, $\C \subseteq \P(X) = 2^X$.  

Of particular significance in geometry is the class of closed, convex sets in $\RR^n$ with the origin in the interior, denoted by ${\mathcal K}_0^n$,  together with the order given by inclusion.  %
Of equal importance are the following two classes of functions that play a central role in analysis and optimization: The class  $\cvx (\R^n)$ of lower semi-continuous convex functions from $\RR^n$ to $(-\infty, \infty]$,  with the order given by point-wise inequality (which corresponds to the inclusion of the epi-graphs of the functions) and the sub-class $\cvx _0 (\R^n)\subseteq \cvx (\R^n)$ of non-negative lower semi-continuous convex functions attaining value zero at the origin with the same order.  %
For each of these examples there exists a duality on the class (not necessarily unique),  which can be extended to be defined on a much larger class, yet on it,  it may no longer be an involution.  We call such a transform an \emph{order reversing quasi involution}:

\begin{definition}\label{def:set-duality}
	Let $X$ be a set, and let $\,T:\P(X)\to \P(X)$, where by $\P(X)$ we denote the power set of $X$.   We will say that $T$ is an  {\bf order reversing quasi involution} if for every $K,L\subseteq X$, the following hold
	\begin{enumerate}[(i)]
		\item $K\subseteq TTK$, \tabto{5.5cm} (quasi involution)
		\item if $L\subseteq K$ then $TK \subseteq TL$.  \tabto{5.5cm}  (order reversion)
	\end{enumerate}
	Given a family of sets $\C \subseteq \P(X)$ and $T:\C\to \C$ which satisfies (i) and (ii) for all $K,L \in \C$ we say that $T$ is  an {\bf order reversing quasi involution on $\C$.}
\end{definition}

It turns out that these abstract order reversing quasi involutions share many interesting structural properties, some of which are usually proven ad hoc in the literature  when an \orqi is discussed. Some of these properties stem from the fact, which is our first theorem in this note, that all \orqis{ } can be identified as a form of ``cost duality for sets'', introduced by the authors in \cite{paper2} in the context of optimal transport. To state the theorem, let us define the \orqi associated with a cost. 

\begin{definition}[Cost duality]\label{def:cost-duality}
	Let  $c:X\times X \to (-\infty,\infty]$ satisfy $c(x,y) = c(y,x)$.   
	For $K\subseteq X$ define the $c$-dual  set of $K$ as
	\[ K^c= \{ y\in X:\, \inf_{x\in K} c(x,y) \ge 0 \}.\]
\end{definition}

Our first main theorem reads as follows.

\begin{thm}\label{thm:duality-induced-by-cost}
	Let $T: \P(X) \to \P(X) $ be an order reversing quasi involution. Then  there exists a cost function $c:X\times X\to \{\pm 1\}$ such that   for all $K\subseteq X$, $TK = K^{c}$.
\end{thm}

A large portion of this paper is dedicated to the study of the inherent properties of  \orqis.  For example,  any \orqi is always an order reversing involution on its image and this image has a natural lattice structure, respected by the quasi involution. In this sense, an \orqi comes always together with its associated class of sets (closed under intersection). Sometimes, as in the case of the polarity transform on convex bodies, the \orqi is unique.

We accompany the general results by many examples of \orqis~ to highlight that one can obtain in this way a vast amount of classical families of sets as well as some less standard ones, which have received attention in recent years,  and also completely  new and interesting examples. 
This extensive array of examples shows how fundamental the notion of an \orqi is, and offers a unifying point of view which deepens our understanding of the underlying principles and structures.   For example, it allows for an integrated approach to the study of invariant sets of a given order reversing quasi involution. Invariant sets are of interest for multiple reasons; we give an example where the invariant sets of an \orqi are exactly the sets of equal width in $\RR^n$, a class studied in convexity with many open questions about it. Often invariant sets have extremal properties within a given family with respect to some functional. In the direction of invariant sets we prove,  for instance, the following theorem.

\begin{thm}\label{thm:inv-sets-with-intersection}
	Let $T:\P(X)\to \P(X)$ be an \orqi and let $X_0 = \{ x\in X: x\in T\{x\}\}$. Let $K_0\subseteq X$ satisfy $K_0 \subset TK_0$. Then there exists some $K \subseteq X$ with $K_0 \subseteq K$ and such that $TK \cap X_0 = K$. In particular, if $X=X_0$ then for any $x_0\in X$ there exists an invariant set $K$ with $x_0\in K$, namely $x_0 \in TK=K$. 
\end{thm}

We also discuss \orqis~ defined on a sub-class $\C\subseteq \P(X)$, and determine necessary and sufficient conditions for such an \orqi to have an extension to $\P(X)$.  This is useful as the existence of an extension,  by means of Theorem \ref{thm:duality-induced-by-cost},  gives the cost representation of the transform and all the resulting properties. To state the theorem we need to define the condition of ``respecting inclusions'', which is a strengthening of order reversion. 

\begin{definition}
	Let $X$ be some set, $\C\subseteq \P(X)$ and  $T:\C \to \C$. 
	We say that $T$ \emph{respects inclusions} if  $K\subseteq  \cup_{i\in I} K_i$ implies $TK \supseteq\cap_{i\in I} TK_i $
	for any $K,K_i \in \C$, $i\in I$. 
\end{definition}
It is easy to see that a mapping respecting inclusions is also order reversing. It is also clear that if $T:\C \to \C$ is an \orqi on $\C$ and it has an extension to an \orqi on $\P(X)$ then $T$ must respect inclusions. Our  next theorem shows that this condition is not only necessary but is sufficient.  

\begin{thm}\label{thm:extending-general-class}
	Let $\C\subseteq \P(X)$ be a family of sets and $T:\C\to \C$ be an \orqi on $\C$ which respects inclusions. Then $T$ can be extended to an \orqi $\hat{T}:\P (X)\to \P(X)$, such that $\hat{T}|_\C =T$. 
\end{thm}  
When the class $\C$ is closed under intersections, any \orqi $T$ on it automatically respects inclusions (see Lemma \ref{lem:extending}).

As an outcome of this line of study, we present ways in which one can build new \orqis~ from existing ones. We explain how to intersect \orqis, and in some cases how to average them. We express some known \orqis~ as intersections of simpler ones. Finally, we show how to construct the \emph{dual} of an order reversing quasi involution, which results in a new quasi involution, as described   in Definition \ref{def:dualorqi}. These methods are illustrated by examples in Section \ref{sec:Zoo}.

Notably, by taking the dual \orqi of the classical polarity transform from convexity, we find a very natural and basic concept which has been hiding out of sight for a long time. In  Section \ref{sec:dual-duality} we present a thorough study of the dual polarity, which turns out to be particularly curious due to its many links with the classical polarity and with other known operations in convexity.  We analyze in detail the ``dual polarity'' given by the cost function $c(x,y)=\sp{x,y}-1$, which enables us to  prove the following theorem.

\begin{thm}\label{thm:newBStypeMAIN1}
	Let $K\subset \RR^{n}$ be such that for some $e\in S^{n-1}=\{x\in \R^{n}: |x|=1\}$ we have that $x+t e \in K$, $x\in e^{\perp}$ implies 
	$-x+te \in K$,  where $|\cdot|$ denotes the Euclidean norm (we then say $K$ is even with respect to this coordinate system). Let ${T: \P(\RR^{n})\to \P(\RR^{n})}$ be given by $TK = \{ x\in \R^n: \iprod{x}{y}\ge 1 \ \forall y \in K\}$. Then 	
	\[ \gamma_{n} (K)\gamma_{n}(TK) \le \gamma_{n}(K_0)^2\]
	where $K_0 = \{ (x,t)\in \RR^{n-1}\times \RR^+: |x|^2 + 1 \le t^2\}$, and $\gamma_{n}$ is the Gaussian measure on $\,\R^{n}$.
\end{thm}

\begin{rem}
	In Section \ref{sec:dual-duality} we analyze $T$ in depth, and, in particular, show that the image of this transform are ``cone-like'' sets (closed convex sets satisfying $\lambda K \subseteq K$ for $\lambda \ge 1$),   that $K_0 = TK_0$, and that $K_0$ is, up to rotation, the only $T$-invariant set  which is even with respect to some coordinate system.  
\end{rem}

\textbf{The paper is organized as follows:} In Section \ref{sec:general-properties} we investigate some basic properties of \orqis. In Section \ref{sec:cost-dualities} we study cost dualities and prove Theorem \ref{thm:duality-induced-by-cost}. In Section \ref{sec:Zoo} we invite the reader into the ``zoo of examples'', where we describe a first selection of illustrative examples that serve as a motivation for further analysis of the general cost transforms. In Section \ref{sec:dual-duality} we discuss in detail one of these examples, which we call ``dual polarity'', together with some new \BS type inequality. 
In Section \ref{sec:fixed-points} we  address the question of fixed points of the transform $K\mapsto K^c$. 
In Section \ref{sec:another-ex} we provide a second collection of  examples, and finally in Section \ref{sec:some additional} we return to the general theory of \orqis~ and describe several  constructions and properties, in particular, the possibility of extension of \orqis, their composition and conjugation, and restriction to a sub-class or a subset.  

\subsection*{Acknowledgments}

The authors were supported in part by the 
European Research Council (ERC) under the European Union’s Horizon 2020
research and innovation programme (grant agreement No 770127). The first named author was supported in part by ISF grant no. 784/20. The second named author is grateful to the Azrieli foundation
for the award of an Azrieli fellowship.

\section{General properties of quasi involutions on sets} \label{sec:general-properties}

We begin by gathering a few simple facts about \orqis .  First, every order reversing quasi involution gives rise to an  order reversing involution on a subset of $\P(X)$. 

\begin{lem}\label{lem:orqi-is-duality-on-image}
	\sloppy Let $T:\P(X)\to \P(X)$ be an order reversing quasi involution and let $~{\C = {\rm Im}(T)= \{ TK: K\subseteq X\} \subseteq \P(X)}$. Then $T:\C \to \C$ is an   order reversing involution. 	
\end{lem}

\begin{proof}
	The fact that $T$ reverses order on $\C$ is inherited from the same property on $\P(X)$.  As for the involution property, it amounts to showing $TTTK = TK$ for any $K\subseteq X$. Using the quasi involution property on $K$, we see that $TTK \supseteq K$, so applying $T$ and using order reversion we get $TTTK \subseteq TK$. On the other hand, we can use the order reversion property on $TK$ instead, which yields, $TK\subseteq TTTK$, and finishes the proof.
\end{proof}

Let us remark that the direction of the inclusion in the quasi involution property $(i)$ is not significant and can be reversed, as there is a one-to-one correspondence between \orqis~ as given by Definition \ref{def:set-duality}, and those defined with the reverse inclusion (see Lemma \ref{lemma:switch-direction}).

A useful property of an order reversing quasi involution $T$ defined on all the subsets of the space $X$ is that it respects the lattice structure, with $\wedge$ corresponding to intersection and $K \vee L$ corresponding to $TT(K\cup L)$, in the sense that $T$ maps unions to intersections (for more on lattice theory, see e.g. \cite{Davey}).  Also, the set $X$ is always the maximal element in the image lattice, and $TX$ (which may or may not be the empty set) is the minimal element in the image lattice. These properties are captured in the following simple proposition, which, nevertheless, will be of much use to us in what follows.

\begin{prop}\label{prop:intersection-and-dual}
	Let $T:\P(X)\to \P(X)$ be an order reversing quasi involution. Then $TTX = X$ and $T\emptyset = X$ and for any collection of sets $K_i\subseteq X$,  $i\in I$,
	\begin{equation}\label{eq:lattice}
	T\left(\cup_{i\in I}K_i\right) = \cap_{i\in I} T(K_i).  
	\end{equation}
\end{prop}

\begin{proof}
	The first two properties are straightforward. Indeed, $TTX \supseteq X$ by the quasi involution condition, so $TTX = X$. Letting $L_0 = TX$, since $\emptyset\subseteq L_0$ we have by order reversion that $T\emptyset \supseteq TL_0 = X$ and so $T\emptyset  = X$. 	  
	
	To show the property \eqref{eq:lattice},  we start by considering	$K_j \subseteq \cup_{i\in I}K_i$ for all $j\in I$. Then by the order reversion property (i),    $T(\cup_{i\in I}K_i) \subseteq T(K_j)$ for any $j$, and, in particular,  $T(\cup_{i\in I}K_i) \subseteq \cap_{j\in I}T(K_j)$.
	
	For the other direction, $\cap_{i\in I} T(K_i) \subseteq T(K_j)$ for any $j\in I$. Applying property (ii) again, $TT(K_j) \subseteq T(\cap_{i\in I} TK_i )$. Using property (i), this implies  $K_j \subseteq T(\cap_{i\in I} TK_i )$, so in particular $\cup_{i\in I}K_i  \subseteq T(\cap_{i\in I} TK_i )$. Applying $T$ duality once more 
	\[ T(\cup_{i\in I}K_i)  \supseteq TT(\cap_{i\in I} TK_i ) \supseteq \cap_{i\in I} TK_i,  \] 
	where for the rightmost inclusion we used property $(i)$ once more.
\end{proof}

\begin{rem}
	It is useful to note that one may repeat the same proof  in the case 
	where  $T$ is defined only on a sub-class $\C\subseteq \P(X)$, and is an  order reversing quasi involution  on ${\C}$. We then get that property  \eqref{eq:lattice} still holds whenever 	$\cup_{i\in I}K_i \in {\C}$.  	\\ 
	In Section 
	\ref{sec:some additional} we discuss \orqis~defined on sub-classes and show that   often one can extend them to be defined on the whole $\P(X)$ (e.g. when  $\C$ is closed under intersections and $X\in \C$, see    Lemma \ref{lem:extending}). 
\end{rem}
It will be useful to note that for any set $K\subseteq X$, the set $TTK$ is the ``envelope'' of $K$, namely the smallest set in the image of $T$ which contains $K$.

\begin{prop}\label{prop:TTKisenvelope}
	Let $T:\P(X)\to \P(X)$ be an order reversing quasi involution, and let $K\subseteq X$. Then 
	\[  TTK = \cap \{ L: L \supseteq K \quad{\rm and}\quad L = TTL\}. \]	
\end{prop}
\begin{proof}
	For any $L = TTL$ with $K\subseteq L$ we have $TK \supseteq TL$ an thus $TTK \subseteq TTL = L$. Therefore, 
	letting $ L_0$ stand for the intersection 	$L_0 = \cap \{ L: L \supseteq K \quad{\rm and}\quad L = TTL\}$ we have $TTK \subseteq L_0$. On the other hand, letting $L = TTK$, we know $K \subseteq L$ and thus $L$ participates in the intersection, which means $L_0 \subseteq L = TTK$ and the proof is complete.  
\end{proof}

\section{Cost dualities}\label{sec:cost-dualities}

A natural source for \orqis, that emerged from authors' previous work \cite{paper2} on optimal transport,   are \emph{cost transforms for sets}.  We recall the definition given in the introduction,  though we include an extra parameter.

\begin{manualdefinition}{1.2}
	Let  $c:X\times X \to (-\infty,\infty]$ satisfy $c(x,y) = c(y,x)$.
	Fix $t\in (-\infty, \infty]$ (which will be omitted in the notation as it is a fixed parameter).   
	For $K\subseteq X$ define the $c$-dual  set of $K$ as
	\[ K^c =\bigcap _{x\in K} \{ y\in X:\, c(x,y)\ge t \}= \{ y\in X:\, \inf_{x\in K} c(x,y) \ge t \}.\]
\end{manualdefinition}

Let us point out (see also Example \ref{ex:polar}) that in the case of the classical cost $c(x,y) = -\iprod{x}{y}$ and $t = -1$, we get the well known polar set $K^\circ$.  Polarity can also be represented by the so-called \emph{polar cost} given by $c(x,y)=-\ln (\sp{x,y}-1)_+$ and $t = \infty$. This cost was introduced in \cite{hila} and studied in \cite{kasia-thesis, paper2}, and is linked with the Polarity transform $\A$ (see Example \ref{exm:A-transf}).  
The classical cost $c(x,y) = -\iprod{x}{y}$ and $t = -5$ gives $K\mapsto 5K^{\circ}$. With $t =  1$, and $c(x,y)=\iprod{x}{y}$ however, we get 
\[K \mapsto \{x: \forall x\in K\,,\, \iprod{x}{y} \ge 1\} \] 
which is very different from polarity, and which we discuss in depth in Section \ref{sec:dual-duality}. We will see many other examples of cost dualities in Section \ref{sec:Zoo}. 

We remark that the notion of $c$-duality may be generalized further.  One may associate with a cost function $c:X\times Y \to (-\infty,\infty]$ the corresponding transform from $\P(X)$ to $\P(Y)$, and its counterpart, from $\P(Y)$ to $\P(X)$.  For simplicity we restrict here to the case $X = Y$ and symmetric $c$.     

Note that very little of the information given by $c$ is used for the definition of the transform.  The choice of $t$ is immaterial as we may exchange $c$ with $c -t$ and take $t=0$, or, if $t = \infty$, we can replace $c$ by $-\exp(-c)$ with $t=0$.  Moreover, two costs $c_1, c_2$, with parameters $t_1, t_2$, produce the same transform if and only if 
\[\{ (x,y): c_1(x,y)\ge t_1 \}  = \{ (x,y): c_2(x,y)\ge t_2 \}.\]  
In other words, the transform is actually defined by this subset of $X\times X$, and the main reason to keep in mind the cost function and the parameter $t$ is that for a fixed cost function and a varying parameter $t$ we get a family of transforms which is sometimes of interest to study as a whole. In what follows, when $t$ is not specified we always mean $t=0$.

\begin{lem} Fix a set $X$ and a cost $c(x,y) = c(y,x):X\times X \to (-\infty, \infty]$. The transform $K\mapsto K^{c}$ is an order reversing quasi involution. 
\end{lem}

\begin{proof}
	Order reversion is immediate from the definition, since if $K\subseteq L$ and $c(x,y)\ge t$ for all $x\in L$ then $c(x,y)\ge t$ for all $x\in K$. For the property $K \subseteq TTK$ note that if $x\in K$ and $y\in K^{c}$ then $c(x,y)\ge t$ so that $x\in K^{cc}$. 	
\end{proof}

\begin{rem}
	As mentioned above, our motivation for studying abstract dualities for subsets arose from our study of optimal transport maps and potentials in \cite{paper2}.  In this setting, we consider a measure space $X$ together with a symmetric cost function $~{c:X\times X\to (-\infty, \infty]}$.  Given a probability measure $\pi$ on $X\times X$  we define the ``total cost'' of $\pi$ by $\int_{X\times X}c(x,y) d\pi(x,y)$.  When this expression is finite, we say that $\pi$ is a finite cost plan between its two marginals, $\mu$ and $\nu$, given by $\mu(A) = \pi (A\times X)$ and $\nu(B) = \pi (X\times B)$, where $A$ and $B$ are any measurable subsets of $X$. 
	
	A necessary condition for the total cost of $\pi$ to be finite is that $\pi$ is concentrated  on a set $S\subseteq X\times X$ such that $c(x,y)<\infty$ for any $(x,y)\in S$.  Studying requirements for the existence of finite cost plans between given measures in \cite{paper2}, a natural construction of the ``$c$-dual'' of a subset of $X$ arose, which, roughly speaking, is as follows: Given a subset $A\subseteq X$, we look at all of the elements $y\in X$ which have finite cost with {\em at least} one element of $A$. The complement of this set is then given by  
	\[ TA = \{ y\in X: \, \forall x\in A, \, c(x,y) = \infty\}. \] 
	
	It is easy to check  that a necessary condition for the existence of a finite cost plan transporting  a probability measures $\mu$ to a probability measure $\nu$ is that  $\mu(A)+\nu(TA)\le 1$ for all $A\subseteq X$. 
	
\end{rem}

Our main observation in this section is that all 
\orqis~ on sets are induced by costs, and in fact, by costs attaining only two values, for example, costs into $\{\pm 1\}$.

\begin{manualtheorem}{1.3}%
	Let $T: \P(X) \to \P(X) $ be an order reversing quasi involution. Then  there exists a cost function $c:X\times X\to \{\pm 1\}$ such that   for all $K\subseteq X$, $TK = K^{c}$.
\end{manualtheorem}

\begin{proof}   
	We shall construct a cost $c$ as required. To do so, define the following subset of $X\times X$ 
	\[ S = \{ (x,y) : y \in T(\{x\})\},\] 
	and define $c(x,y) = 1$ if $(x,y)\in S$ and $-1$ elsewhere. Note that $S$ is a symmetric set (that is, $(x,y) \in S$ implies $(y,x) \in S$): Indeed, 
	$y\in T(\{x\})$ implies $T(\{y\}) \supseteq TT(\{x\}) \supseteq \{x\}$
	so that $x\in T(\{y\})$.

	Let $K\subseteq X$, we claim that $TK=K^c$. 
	Note that 
	\[K^c = \{y:\ \forall x\in K\ c(x,y)\ge 0\} 
	= \{y:\ \forall x\in K\ (x,y)\in S\}
	=  \{y:\ \forall x\in K\ y\in T(\{x\})\}.\]
	Thus,
	\[K^c\subseteq \cap_{x\in K} T(\{x\}) = T(\cup_{x\in K} \{x\}) = TK. \]
	
	Next let $y\in TK$, then $T(\{y\}) \supseteq TTK \supseteq K$ so that for every $x\in K$ we have $x\in T(\{y\})$ which, by the symmetry of $S$ mentioned above, means that for every $x\in K$, $y\in T(\{x\})$ so that 
	for every $x\in K$, $c(x,y) \ge 0$ namely $y\in K^c$ as required. 
\end{proof}

As the proof above demonstrates, to every
symmetric (with respect to $(x,y)\mapsto (y,x)$) subset $S\subseteq X\times X$ there corresponds a unique order reversing quasi involution. Indeed,  given $S$ we define a cost by $c(x,y) = 1$ if $(x,y)\in S$ and $-1$ elsewhere. Let
\[ T\{x\}= \{x\}^c = \{y : c(x,y) \ge 0\} =\{y: (x,y) \in  S\} = : S_x\]
(which can be called the fiber of $X$). For any $K\subseteq X$ 
\[ TK = T(\cup_{x\in K} \{x\}) = \cap_{x\in K}S_x.\]
Clearly different sets $S$ produce different \orqis~ $T$, since at least one of the fibers is different.  

In what follows we will associate an \orqi $T$ with either a cost function $c$ or a set $S\subseteq X\times Y$, according to what is more convenient for the clarity of the proof.

\begin{rem}\label{rem: set S}
	It will be useful to have another representation of the set $S$ defined in the proof of Theorem \ref{thm:duality-induced-by-cost} above. This set can be equivalently written as \[ \tilde{S} = \cup \{   TK \times TTK: K \subseteq X\}\subseteq X\times X.\] 
	Indeed,  we have that $S=\cup_x \{x\}\times\{Tx\} = \cup_x \{Tx\}\times\{x\} \subseteq \cup_x \{Tx\}\times TT\{x\}\subseteq \cup_K TK \times TTK$. For the reverse inclusion $\tilde{S}\subseteq S$ note that if we choose a point $(x,y)$ in $\tilde{S}$ then there is some $K\subseteq X$ such that $(x,y)\in TK\times TTK$ and it follows that since $x\in TK$ we have that $TTK\subseteq T\{x\}$ and hence $y\in T\{x\}$, which means that $(x,y)\in S$. 
\end{rem}

\section{A Selection of Examples}\label{sec:Zoo}

Many of the classical operations considered in convexity theory are \orqis.  In this section we gather what we find to be a selection of illuminating examples some of which are well known, some less known and some completely new.  Another collection will be presented in Section \ref{sec:another-ex}.

\begin{exm}[Polarity]\label{ex:polar}
	Consider the polarity transform $T:\P(\R^n)\to \P(\R^n)$ given by 
	$$TK= K\dual=\{y:\forall x\in K\ \iprod{x}{y}\le 1\}.$$
	The associated set is $S=\{(x,y): \, \sp{x,y}\le 1\}$. The image class for this transform is 
	${\mathcal K}_0^n$, the class of closed convex sets which include the origin. 
	To write it as a cost-transform, one may take $c(x,y)=-\iprod{x}{y} +1$ so that  $K^c=\{y: \forall x\in K \ -\sp{x,y}+1\ge 0\}=K^\circ$. 
	As is well known, $K^{\circ\circ} = \conv(K, 0)$, the smallest set in the class which includes $K$.  The only invariant set is $B_2^n=\{x\in \R^n: \, |x| \le 1\}$ (we discuss invariant sets for \orqis~ in Section \ref{sec:fixed-points}).  
	In \cite{boroczky-schneider}, B\"or\"oczky and Schneider showed that polarity is essentially the only order reversing involution on ${\mathcal K}_0^n$.  
	For an overview of polarity in classical convexity see, for example,  \cite[Section 1.6]{schneider-book}.

\end{exm}
Another known and important duality, which we present in the setting of \orqis, is the Legendre transform $\L$ which maps the class $\cvx(\R^n)$ (consisting of proper lower semi continuous convex functions $\varphi: \RR^n\to \RR\cup \{+\infty\}$)  to itself.  We identify a function with its epigraph, and then the transform can be extended to  an \orqi on $\RR^{n+1}$ with a natural associated cost function.

\begin{exm}[Legendre transform]\label{exm:Legendre}
	Fix a partition $\RR^{n+1} = \RR^n \times \RR$. Given a subset $A\subseteq \RR^{n+1}$, complete it to be a set $A'$ which is an epi-graph of a function $\varphi: \RR^{n}\to [-\infty, \infty]$ (letting $(x,s)$ belong to $A'$ whenever $(x,t)\in A$  and $s\ge t$). Consider the transform $T: \P(\RR^{n+1})\to \P(\RR^{n+1})$ defined by $TA = \epi(\L\varphi)$.  
	Here $\L$ denotes 
	the Legendre transform, 
	\[ \L \varphi (y) = \sup_x \left(\iprod{x}{y}-\varphi(x) \right),\]
	and $\epi(\varphi) = \{(x,t): t\ge \varphi(x)\}$.	The associated set is 
	\[ S = \{ ((x,t), (y,s)) :   \iprod{x}{y} \le s+ t \}.\]
	
	The image class for this transform is the class of epi-graphs of functions in  $\cvx(\RR^n)$ together  with the constant $+\infty$ and the constant $-\infty$ functions. It is easy to check that $T$
	is an \orqi using, in particular, the fact that $\L\L\varphi \le \varphi$ for all $\varphi\in \cvx(\RR^n)$. 
	To write it as a cost transform, one may take $c((x,t), (y,s)) = t+ s - \iprod{x}{y}$. The only invariant set is $\epi(\|x\|^2/2)$, see  \cite[Section~9.2.5]{AGAbook2}.  
\end{exm}

The Legendre transform is just one of many functional cost transforms, all of which fit into this framework in a similar manner, see Example \ref{ex:functional-cost-trans}. 

Our next example is in the setting of metric spaces, and does not require a linear structure on the space $X$. 

\begin{exm}[Complements of neighborhoods]\label{ex:complement-of-tneighborhood}
	Consider the   transform $T: \P(X)\to \P(X)$ where $(X,d)$ is a metric space, given by 
	\[ TA = \{ y: d(x,y)\ge \eps \ \forall x\in A\}, \] 
	which corresponds to taking a set to the complement of its $\eps$-neighborhood. 	It is clear that $A\subseteq TT A$ and that $T$ reverses order. 
	The associated set is  $S = \{ (x,y): d(x,y)\ge \eps\}$. The image class for this transform consists of complements of unions of $\eps$-balls. For example, all convex sets are in the class. To write $T$ as a cost transform, one may take $c(x,y) = d(x,y) - \eps$. Clearly, there are no invariant sets. 
\end{exm}

The next example is quite trivial, but will show up unexpectedly in a few pages; we emphasize that {\em any} choice of symmetric cost function induces an \orqi and generates a class of sets, so one may view this example as simply ``experimenting" with artificial cost functions. 

\begin{exm}[Balls]\label{ex:balls}
	Consider the cost function $c(x,y)=1-\|x\|\|y\|$ on $\RR^{n}$ where $\|\cdot\|$ is any fixed norm.  Equivalently, consider 
	\[ S = \{ (x,y): \|x\|\|y\|\le 1\}.\]
	Then 
	\[ TA = \{ y: \forall x\in A\ \|y\| \le 1/\|x\|\} = \frac{1}{R} K \]
	where $R = \sup \{ \|x\|: x\in A\}$ and $K$ is the unit ball of the norm $\|\cdot\|$, i.e. $K=\{x \in \R^n: \ \|x\|\le1 \}$. 
	The image class consists of positive multiples of $K$, as well as $\{0\}$ and $\RR^n$, and $K$ is the only invariant set in the class.
	
\end{exm}

In the next example the induced class is that of complements of ``flowers''.  A flower is defined as a union of the form $A=\bigcup_{x\in C}B\left(\frac{x}{2},\frac{|x|}{2}\right)$, where $C\subseteq \RR^n$ is some closed set.
The class of flowers is closed under unions (not intersections), which explains why we need to consider complements (for the relation between a class and the class of complements see Lemma \ref{lemma:switch-direction}). The class of flowers was recently investigated in \cite{flowers1, flowers2}, and the flower transform $K \mapsto K^\clubsuit$, mapping a convex body to its ``flower'',  was defined to be the star shaped set with radial function $r_{K^\clubsuit}=h_K$.

\begin{exm}[Flowers]\label{ex:flowers}
	Consider the set $S\subseteq \R^n \times \R^n$ given by 
	\[ S = \{  (x,y): \iprod{x}{y} < \frac12 |x|^2|y|^2        \}. \] 
	The associated transform maps a set $A$ to the set 
	\[ TA= \R^n \setminus \bigcup_{x\in A} B(\tfrac{x}{|x|^2}, \tfrac{1}{|x|}). \]
	The image class is thus clearly (using that $\Phi(x)=x/|x|^2$ is a bijection) all complements of flowers. In terms of the flower transform we have $T(\Phi(K/2)) = \R^n \setminus K^\clubsuit$. 
	The image of a single point $x_0\in X$ is given by  
	\[ \{ y: 2\iprod{z_0}{y} < |y|^2  \} = \RR^n \setminus B(z_0, |z_0|)\]
	where $z_0 = x_0 /|x_0|^2$. We discuss in Section \ref{sec:fixed-points} the easily verifiable fact that there is a unique invariant set for $T$, namely the set $\{ x: |x|\ge \sqrt{2}\}$. 
\end{exm}

We shall use the next example to demonstrate a way to use given \orqis~to build new ones by way of intersection. We start by presenting an \orqi whose image is the class of reciprocal convex sets, 
studied by Milman, Milman and Rotem \cite{flowers1}. We then show that it can be obtained as an intersection of a family of simple \orqis.

\begin{exm}[Reciprocals]\label{ex:recip} 
	Consider the set $S\subseteq \R^n \times \R^n$ given by 
	\[ S = \{  (x,y):   \iprod{x}{\theta}\iprod{y}{\theta}  \le 1      \ \forall \theta\in S^{n-1}  \}. \] 
	One may compute the associated transform
	\begin{eqnarray*}
		T A  %
		& = & \{ y\in \RR^n: \iprod{x}{\theta}\iprod{y}{\theta}  \le 1\ \forall x\in A\ \forall \theta\in S^{n-1}    \}\\
		& = & \{ y\in \RR^n: 
		(\sup_{x\in A}\iprod{x}{\theta})   \iprod{y}{\theta}\le 1\
		\forall  \theta\in S^{n-1}    \}\\
		& = & \{ y\in \RR^n: h_A(\theta)\iprod{y}{\theta} \le 1\ \forall \theta\in S^{n-1}   \}.	
	\end{eqnarray*}
	
	Note that when $h_A(\theta) <0$ then, as $h_A(-\theta) \ge -h_A(\theta)$, we may ignore, in such cases,  the condition $h_A(\theta)\iprod{y}{\theta} \le 1$. In fact, one may easily check that $TA = T(\conv(0,A))$.  
	The image class of $T$ is the class of so-called ``reciprocal bodies'', namely the Wulff shapes (or Alexandrov bodies) associated with the functions $1/h_A$ (on the sphere), where $A\in \K_0^{n}$, i.e. the set
	\[\{ y\in \RR^n: \forall   \theta\in S^{n-1}\      
	\iprod{y}{\theta} \le 1/h_A(\theta) \}.\] 
	
	\sloppy The transform itself is the ``reciprocal transform'' considered in \cite{flowers1}, denoted there by $~{TK = K'}$. The only invariant set is the euclidean unit ball. A natural cost to consider is of course $c(x,y) = 1- \sup_{\theta \in S^{n-1}} \iprod{x}{\theta} \iprod{\theta}{y}$. 
\end{exm} 

It turns out that many properties of the reciprocal transform shown in \cite{flowers1} follow immediately from the theory of \orqis{}.  Firstly,  by Lemma \ref{lem:orqi-is-duality-on-image} we have that $K'''=K'$.  Secondly,  it follows that $K'\subseteq K^\circ$.  Indeed, 
we have \[ c(x,y) = 1- \sup_{\theta\in S^{n-1}}\iprod{x}{\theta}\iprod{y}{\theta} \le  1- \iprod{x}{y/|y|}\iprod{y}{y/|y|} = 1- \iprod{x}{y} =: c_2(x,y), \] 
and since $c_2$ is the cost inducing polarity, the claim follows.

Additionally, one may present this example in multiple interesting ways, for example, we show that it can be viewed as an intersection of \orqis.  In terms of costs, ``intersection'' corresponds to taking the minimum of the given costs, and in terms of the associated sets $S$, it is simply their intersection. 

\begin{fact}\label{prop-intersection}
	Assume $T_1$ and $T_2$ are two order reversing quasi involutions. Then also $ T_3$, defined by $T_3A = T_1 A \cap T_2 A$ is an order reversing quasi involution. Similarly, if $T_\alpha$ is an \orqi for any $\alpha \in I$, then $TA = \cap_{\alpha\in I} T_\alpha A$ is an \orqi as well. 
\end{fact}

Therefore, one can think of Example \ref{ex:recip} as an intersection of a continuum of the following simple order reversing quasi involutions:

\begin{exm}[$\theta$-slabs]
	\sloppy For a fixed $\theta\in S^{n-1}$ consider the cost function $c_\theta$ on $\R^n\times \R^n$ given by ${c_\theta(x,y) = 1- \iprod{x}{\theta}{\iprod{y}{\theta}}}$.  The associated set is $S_\theta=\{(x,y): \ \iprod{x}{\theta}{\iprod{y}{\theta}} \le 1 \}$. The image class consists of slabs and half-spaces, containing the origin, in direction $\theta$. More precisely, denoting the transform $T_\theta$, we have 
	\[ T_\theta A  = \{ y:  \iprod{y}{\theta} \le 1/\iprod{x}{\theta} \ \forall x\in A \},	
	\]	
	which simply means that $T_\theta A = (P_{\theta \RR} A)^\circ$. 
	
	As $\cap_{\theta \in S^{n-1}} S_\theta $ gives the set $S$ from Example \ref{ex:recip}, we see that the reciprocal transform is an intersection of these simple slab-transforms. 
	
\end{exm}

This representation again offers a new insight into the geometry of reciprocals,  namely that \[K' = \cap_\theta (P_\theta K)^\circ.\]

We proceed to  yet another representation of the reciprocals, which places the family of reciprocal bodies as an average of polarity presented in Example \ref{ex:polar} and balls from Example \ref{ex:balls}. More precisely, we consider a continuous family of \orqis~ by means of averaging two cost functions.  Undoubtedly one can average any two costs to get new examples, however, since a transform does not have a uniquely associated cost, the family of ``intermediate transforms'' is not uniquely defined. Nonetheless, when there is a natural cost associated with an order reversing quasi involution, some interesting examples can be produced.  As already mentioned, our example will be that of the reciprocal transform expressed as the arithmetic average of polarity (Example \ref{ex:polar}) and balls (Example \ref{ex:balls}).

\begin{exm}[Reciprocal-type transform] 
	Consider a family of costs on $\R^n \times \R^n$ given by $c_\l(x,y)= 1- (1-\l)|x||y| +\l\sp{x,y}$.  The associated sets are  given by 
	\[S_\l = \{(x,y):  (1-\l)|x||y| +\l\sp{x,y} \le 1\}\]
	The image class of the associated transform depends, of course,  on $\l$.  In the case $\l=\tfrac 12$, as mentioned above, the image class consists of reciprocals convex bodies, as in Example \ref{ex:recip}. Indeed, 
	\begin{align*}
	\sup_{\theta \in S^{n-1}} \iprod{x}{\theta}\iprod{\theta}{y} &= \sup_{\theta\in S^{n-1}} |x| |y| \cos ( \sphericalangle \theta x) \cos ( \sphericalangle \theta y) \\ &= \tfrac{1}{2} |x||y| \sup_{\theta\in S^{n-1}} ( \cos (\sphericalangle \theta x +\sphericalangle \theta y) +\cos  (\sphericalangle \theta x - \sphericalangle \theta y) ) \\
	&= \tfrac{1}{2} |x||y| ( \cos (\sphericalangle xy)+\cos(0)) = \tfrac 12 ( \sp{x,y}+|x||y|)
	\end{align*}
	where the supremum is obtained at $\theta = \frac{1}{2}(\hat{x} + \hat{y})$ where $\hat{v} = v/|v|$. 
\end{exm}

In the last part of this section we present a new construction, which produces a ``dual'' of a given order reversing quasi involution. That is, \orqis~ come in pairs.  Moreover, this pairing reverses the natural order on \orqis~ (on a given space $X$) given by the pull back of the inclusion order on the associated sets $S\subseteq X\times X$.

As we have seen in Section \ref{sec:cost-dualities},  every \orqi $T:\P(X)\to \P(X)$ is determined by the 
set $S=\{TK\times TTK: \, K\subseteq X\}\subseteq X\times X$.  In many cases (e.g. all of the cases above except for Example \ref{ex:flowers}) the set 
$S$ is  closed, which ensures that the image class consists of closed sets. This is a desirable property (usually) which is why, below, we chose to define the pairing in a way that includes a closure operation (we let $\overline{A}$ stand for the closure of $A$ with respect the standard topology). 

\begin{definition}\label{def:dualorqi}
	\sloppy Given a topological space $X$ and an \orqi  $~{T:~\P(X)~\to~\P(X)}$ with an associated set $S_T =\{TK\times TTK: \, K\subseteq X\}\subseteq X\times X $ we define its {\bf dual order reversing quasi involution} to be $T':\P(X) \to \P(X)$ with an associated set $S_{T'} = \overline{X\times X\setminus S_T}$ (here closure is taken with respect to the product topology). 	
\end{definition}

Since $S_T$ is symmetric, so is $S_{T'}$. When $c$ is ``reasonable'', the cost associated with $T'$ is simply $-c_T$ where $c_T$ is the cost associated with $T$. We next dualize two of the above examples. First, consider the dual to Example \ref{ex:complement-of-tneighborhood}.

\begin{exm}[Ball intersections]\label{ex: dual-neighborhoods} Let $(X,d)$ be some metric space. 
	Let
	\[S = \{ (x,y): d(x,y)\le \eps\}.\]
	The associated transform is given by 
	\[ 
	T A = \cap_{x\in A} B(x,\eps). 
	\] The image class consists of all sets obtained by intersections of balls of radius $\eps$.  In particular, these sets are closed and of diameter at most $2\eps$. When $X$ is a linear space, the sets in the image class are convex, and the transform is shift invariant $T(a+ A) = a + TA$ (which incidentally happens for a transform associated with a set $S$ if and only if $S = S+D$  for the diagonal $D = \{ (x,x): x\in X\}$). An exciting feature of $T$ is that its invariant sets are the so-called ``diametrically complete'' sets, and when $X = \RR^n$ with the Euclidean distance $d$, these are precisely sets of equal width $\eps$. 
	The analysis of invariant sets of \orqis~in general, and of $T$ in particular, will be presented in Section \ref{sec:fixed-points}.
\end{exm}

The next example is extremely rich, and we devote a full section to it (Section \ref{sec:dual-duality}). It is the dual (in the sense of Definition \ref{def:dualorqi}) of the classical polarity from Example \ref{ex:polar}.  

\begin{exm}[Dual polarity]\label{ex:dual-polarity}
	Let
	\[ S= \{(x,y): \ \sp{x,y}\ge 1\}\subseteq \RR^n \times \RR^n. \]
	The associated transform is given by 
	\[ TA = \{ y: \iprod{x}{y} \ge 1\ \forall x\in A\}.\] 
	The image class   consists of intersections of affine half-spaces that do not include the origin.  In particular, these are unbounded, closed and convex sets.  We analyze this class and the transform in detail in the next section. 
\end{exm}

\begin{rem}
	This section included just a few of the many interesting examples for \orqis. We included more examples in Section \ref{sec:another-ex}. We mention that the dual of the `Legendre-induced' \orqi from Example \ref{exm:Legendre} remains the same type up to a change of coordinate system since 
	\[ S_{T'} = \{ ((x,s),(y,t)): \iprod{x}{y} \ge s+ t \} = 
	\{ ((x,s),(y,t)): \iprod{-x}{y} \le -s- t \}
	\]
	is a linear image of $S_T= \{ ((x,s),(y,t)): \iprod{x}{y} \le s+ t \}$. 
\end{rem}

\begin{rem} 
	Let us mention that while many geometric families of sets appear as images of \orqis~  or their complement, some classical families of convex bodies cannot arise in this way. For example, the classes of Zonotopes (Minkowski sums of a finite number of segments) or Zonoids (limits of Zonotopes) cannot be the image class of an \orqi as they are not closed under intersection (nor can their complements, as they are not closed under unions either).
\end{rem}

\section{Dual Polarity}\label{sec:dual-duality}

The classical polarity transform plays a central role in convex geometry (see e.g. \cite{schneider-book}). Its dual, given in Example \ref{ex:dual-polarity}, is the object we carefully analyze in this section. Throughout this section, $~{T:\P(\RR^n) \to \P(\RR^n)}$ will denote the \orqi given by 
\[ TA = \{ y: \iprod{x}{y} \ge 1\ \forall x\in A\}.
\]

\subsection{The image of $T$ and a connected functional transform}
By definition, the image of the transform consists of intersections of affine half-spaces that do not include the origin. We can characterize the image in several useful ways.   

To begin with,  note that the image of a point $\{x_0\}$ is the halfspace 
\[ H_{x_0} = \{ y:\iprod{x_0}{y}\ge 1 \}.\] 
By Proposition \ref{prop:intersection-and-dual},  the  image class consists of all possible intersections of such half-spaces. In particular, if $K$ is in the image class,  then $K$ is convex and 
$\lambda K \subseteq K$ 
for any $\lambda >1$, a property one might call ``cone-like''. Together with the property of not including the origin, this characterizes sets in the image.  
\begin{lem}\label{lem:class-of-dual-polarity}
	The  class  $\C = \{ TK : K\subseteq \RR^n\}$ consists of \,$\RR^n$ together with all closed convex sets $K\subseteq \RR^n$ that do not include the origin and satisfy for all $\lambda\ge 1$ that $\lambda K \subseteq K$. 
\end{lem}

\begin{proof} 
	Let us denote 
	\begin{align*}
	&\C_1 = \{ K\subseteq \RR^n : \exists (u_i)_{i\in I\neq \emptyset } \quad{\rm and} \quad K = \cap_{i\in I} \{ x: \iprod{x}{u_i} \ge 1 \}, \\
	&\C_2 = \{ K\subseteq \RR^n  :  K {\rm~is~closed,~convex}, 0\notin K, \ \lambda K \subseteq K {\rm ~for~all~}\lambda \ge 1 \}.	
	\end{align*} 
	
	We need to show that the two are equal. The fact that $\C_1\subseteq \C_2$ is obvious since the intersection of closed sets remains closed and the other cone-like properties are preserved by the intersection. For the other direction, given $K\in \C_2$, using that it is closed and convex, there exists a collection of vectors $v$ and constants $\eta(v)$ such that 
	\[ K  = \cap_v \{ x: \iprod{x}{v} \ge \eta(v)	
	\}.\]
	When $\eta(v) >0$, we consider $u = v/\eta(v)$. We claim that the condition  $\lambda K \subseteq K $ for $\lambda >1$ implies that one may dismiss the pairs $(v,\eta(v))$ for $\eta(v)<0$. Indeed, since the set $~{\{ x: \iprod{x}{v} \ge \eta(v)\}}$ is monotone in $\eta(v)$, we may assume that for a given $v$, we chose the maximal among all  constants $\eta(v)$ such that $K \subseteq \{ x: \iprod{x}{v} \ge \eta(v)	
	\}$. Therefore, assuming towards a contradiction that $\eta(v)<0$ means that    
	$K \not\subseteq \{ x: \iprod{x}{v} \ge \eta(v)/2	
	\}$, and so there is some $x_0\in K$ with $\iprod{x_0}{v} < \eta(v)/2$. However, then $3x_0$ (which must also be in $K$ as $3K \subseteq K$) satisfies $\iprod{3x_0}{v} < 3\eta(v)/2<\eta(v)$ contradicting  $K \subseteq \{ x: \iprod{x}{v} \ge \eta(v)	
	\}$. 
	
	Finally, we need to consider the case $\eta(v) = 0$. We claim that if $\{ x: \iprod{x}{v} \ge 0\}$ participates in the intersection defining $K$, we may replace it by an intersection of the form 
	$~{\cap \{ \iprod{x}{w_t} \ge \eta_t>0\}}$, in the sense that 
	$K \subseteq 	\cap \{ \iprod{x}{w_t} \ge \eta_t>0\}\subseteq \{ x: \iprod{x}{v} \ge 0\}$. To define $w_t$ and $\eta_t$, consider some hyperplane $\{ x: \iprod{x}{w} = \eta_1\}$ separating the origin and $K$. Consider $w_t = (1-t) v+ t w$ and $\eta_t = t\eta_1$.  Then we have 
	\[ K \subseteq \{ x: \iprod{x}{w_t} \ge t\eta_1\}\]
	and on the other hand if $	\iprod{x}{w_t} \ge t\eta_1$ for all $t\in (0,1)$ then
	$	(1-t)\iprod{x}{v}   \ge t(\eta_1-\iprod{x}{w})$ so taking $t\to 0$ we see	
	$\iprod{x}{v}\ge 0$. This completes the proof.  
\end{proof} %

The   class  $\C$ decomposes into sub-classes $\C_u$, $u\in S^{n-1}$, which are invariant under $T$, in the following way: Every set $K\in \C$ (except for the empty set and all of $\RR^n$) contains a unique point $0\neq x\in K$ which is closest to the origin. Indeed, it contains at least one such point by closedness, whereas convexity implies the uniqueness. 
The sub-class ${\mathcal C}_u$ consists of those $K\in \C$ whose closest point to the origin lies on the ray $u\RR^+$. We notice that ${\mathcal C}_u$ is a rotation of ${\mathcal C}_{v}$ for any two directions $u,v\in S^{n-1}$ and that the mapping $T$ respects this filtration. 

\begin{lem}\label{lem:invariant_subclass}
	For any $u\in S^{n-1}$ the sub-class $\, \C_u$ is invariant under $T$.	
\end{lem}

\begin{proof} %
	For $K\in \C_u$ we let $H$ denote the halfspace of the form $\{x: \iprod{x}{u} \ge a\}$ and $\ell$ the ray $\{ \lambda u: \lambda \ge a\}$, where $a>0$ and $au$ is the  point in $K$ closest to the origin. Then $\ell \subseteq K \subseteq H$, and so $TH\subseteq T K\subseteq T\ell$. A direct computation gives that $T\ell = \{x: \iprod{x}{u} \ge 1/a\}$ and $TH = \{ \lambda u: \lambda \ge 1/a\}$. Therefore, $TK$ belongs to $\C_u$ as well. 	
\end{proof}%

As is clear from the above proof, $\C_u = \cup_{a>0} \C_{u,a}$ where $K\in \C_{u,a}$ if its point closest to the origin is $au$, and further, it follows that the transform $T$ maps $\C_{u,a}$ to $\C_{u,1/a}$. 
In addition, the mapping $K \mapsto \lambda K$ is a  bijection between $\C_{u,a}$ and $\C_{u, \lambda a}$.

Therefore, having fixed an orthonormal basis $\{ e_i\}_{i=1}^n$ for $\RR^{n}$, in order to study $T$ it suffices to focus on one sub-class ${\mathcal C}_{e_{n},1}$. We call this sub-class ${\mathcal S}:={\mathcal C}_{e_{n},1}$. In the next lemma we demonstrate that one can identify $\S$ with a class of functions and we provide with a functional transform that corresponds to $T$.  

For $\varphi: \RR^{n-1}\to (0, \infty]$ we let $\epi(\v) = \{ x + te_n: t\ge \v(x), \ x\in e_n^\perp\} \subset \R^n$, where we have chosen an orthonormal basis so as to identify  $e_n^\perp$ and $\RR^{n-1}$.

\begin{prop}\label{prop:functional-form-or-dp}
	Let $K\in \S$ and let  $\v: \RR^{n-1}\to (0, \infty]$ be given by 
	\[ \varphi(x) = \min \{ t: (x,t) \in K \}.\]
	Let \[ {\tilde{T}}\v (y) = \sup_{x\in \R^{n-1}} \frac{1-\sp{x,y}}{\v(x)}.\]
	Then $\epi(\v) = K$ and  
	$TK = \epi(\tilde{T}\varphi)$.
	
	The identification $K\mapsto \v$ is a bijection between $\S$ and the class of %
	lower semi continuous convex functions on $\RR^{n-1}$ with  minimal value $\v(0)= 1$, and satisfying $-1\le \L \v|_{\dom (\L \v)} \le 0$.
\end{prop}

\begin{proof}
	Clearly
	$\varphi$ defined above is convex, attains its positive minimal value $1$ at the origin, and 
	\[  \epi(\varphi) = \{(x, t) \in \R^{n-1}\times \R:  \v (x)\le t\} = K \]
	as $K$ is closed and $(x,t)\in K$ implies $(x,s)\in K$ for any $s>t$, which follows from the fact that $~{(x,s)\in \overline{\conv}([1,\infty)e_n \cup \{(x,t)\})}$. We can thus compute  
	\begin{align*}
	T K &= \{(y,t)\in \R^{n-1}\times \R: \, \sp{x,y} +\l \v (x) t\ge 1 \ \ \forall x\in \R^{n-1},  \ \forall \lambda\ge1\}\\  &=  \{(y,t)\in \R^{n-1}\times \R: \, t \ge \frac{1-\sp{x,y}}{\lambda \v (x)}\ \ \forall x\in \R^{n-1},  \ \forall \lambda\ge1\}.
	\end{align*}
	
	Note that in fact we must have $t \ge \frac{1-\sp{x,y}}{\v(x)}$ and hence, we may  define the order reversing transform $\tilde{T}$ on functions as in the statement of the Proposition. 
	Clearly, this $\tilde{T}$ is an order reversing bijection on functions associated with sets in $\S$ 
	(and since at $x=0$ the value of the expression in the supremum is $1$, it is enough to consider the supermum over all $x$ such that $1\ge \sp{x,y}$). 
	
	We are thus left with proving the characterization for functions associated with sets in $\S$. 
	\sloppy To this end let $\v$ be a function associated with some $K \in \S$ and assume that $~{y\in \dom (\L \v)\subseteq \R^{n-1}}$.  Let $v=(-y,1)\in \R^n$ and let $c(v)=\max \{c: \, K\subseteq \{z: \, \sp{z,v}\ge c\}\}$. We claim that $c(v)>-\infty$ if and only if $y\in \dom(\L\v)$, that is, $\L\v (y) <\infty$. Indeed, 
	\[ \L\v(y)  = \sup_x (\sp{x,y}-\v(x)) = \sup_x \sp{(x, \v(x)),(y, -1)}   = \sup_{(x,t)\in \epi(\v)} \sp{(x, t),-v}  \]
	so that $\inf_{z\in K} \sp{z,v} = -\L\v(y)$.  
	From the proof of  Lemma         
	\ref{lem:class-of-dual-polarity}
	we see that either $c(v)\ge 0$ or $c(v)=-\infty$ (in which case $v$ is not to be chosen in the collection of half-spaces intersected), which means precisely that $\L\v(y)\le 0$ on its domain. Clearly, $-\L\v(y) = c(v)\le 1$ since $e_n \in K$ and $\sp{e_n,v} = -1$.

	For the opposite inclusion,  let $\v$ be a convex function such that $-1\le \L \v|_{\dom (\L \v)} \le 0$ and $\min \v (x) = \v(0)=1$. As we have computed, $(x,y)\in \epi(\v)$ implies $\iprod{(x,t)}{(y,-1)}\le \L\v (y)$ for all $y\in \dom(\L \v)$, and in the other direction, if $\iprod{(x,t)}{(y,-1)}\le \L\v (y)$ for all $y$ then $t\ge \L\L \v(x)=\v(x)$, so we have 
	\[\epi(\v) = \cap \{(x,t): \iprod {(x,t)}{(y,-1)}\le \L\v (y), \forall y\in \dom(\L\v)\}.\] 
	It then follows from Lemma \ref{lem:class-of-dual-polarity} that $\epi(\v)$ belongs to $\C$. Since the unique minimum of $\varphi$ is 1 and is attained at the origin, we can conclude that in fact $\epi(\v)$ belongs to $\S$, completing the proof. 
\end{proof}

\begin{rem}\label{rem:Liran-duality}
	It turns out that a variant of the  transform $\tilde{T}$  has been considered, as a special case, by Rotem \cite{Liran-sharpBS}. There, a transform $\sharp$ is defined as
	
	\[f^\sharp(x) = \frac{1}{\sup_y \left[f(y) \left( 1+ \frac{\iprod{x}{y}}{\beta}\right)^\beta \right]}.\]
	
	Clearly, $\left(\frac{1}{\v}\right)^\sharp (-x) = \frac{1}{\tilde{T}(x)}$. The a priori small modification of taking a reflection of $x$ in the definition of $\sharp$ will be crucial when discussing invariant sets (see Section \ref{sec:fixed-points}). 
\end{rem}

\subsection{The $J$ transform as the link with polarity}

It turns out that the transform $T$, when restricted to the subclass $\C_{e_n,1}\subset \P(\R^{n})$, say, is, up to a minus sign, the pull-back of the classical polarity transform under a point-map on the upper half-space. This point map is the one associated with the order preserving transform $J$ defined in \cite{ArtsteinMilmanHidden} and studied intensively in \cite{artstein-florentin2012order}. We illustrate this connection in this subsection. However, one should bare in mind that the mapping $J$ is defined only on subsets of a half space, and when considered as transforms on the whole space $\P(\RR^n)$ the polarity and the dual-polarity are not conjugate, and are in fact quite different. 

Recall that the $J$ transform on subsets of $\RR^{n-1}\times \RR^+$ is defined as $JK = F(K)$, where $F(x,t) = (x/t, 1/t)$. This is a convexity preserving map that maps rays emanating from the origin to rays parallel to the ray $\{0\} \times \RR^+$. 
Consider a set $K\in \C_{e_n,1} \subset \P(\RR^n)$. We already showed that it consists of rays, $K  = \{ t(x,1) : t\ge a(x) \}$ where $a: \RR^{n-1} \to [1,\infty]$.  Note that $a(0)  = 1$ and when $a(x) = \infty$ this means the ray $\RR(x,1)$ does not intersect $K$.

\def\Lbod{{\tilde{J}}}

Thus $JK = \{ (x, \frac{1}{t}): t\ge a(x)\} = \{ (x,s) : s\le \frac{1}{a(x)}\}$. Since $K$ was convex so is $JK$, hence $1/a(x)$ is a concave function on its support (which consists precisely of those $x$'s such that the ray $(x,1)\RR$ intersects $K$ or is asymptotically tangent to it).
In this way we can associate to every $K \in \C_{e_n,1}$ a convex body in $\RR^{n}$ given by 
\[\Lbod(K) = \{(x,t): |t|\le1/a(x) \}.\]  
Note that this operation is simply applying $J$ to $K$ and taking a union with the reflection, i.e. $\Lbod(K)= JK \cup R_{e_{n}}JK$ where $R_{e_n}y = y -2\iprod{e_n}{y}e_n$. The resulting body always includes the segment $[-e_n,e_n]$, is included in the slab $\{ |\iprod{\cdot}{e_n} |\le 1\}$, and is invariant under reflections about $e_n^\perp$. 

\begin{lem}\label{lem:TisJ}
	For $K\in \C_{e_n,1}\subset \P(\RR^{n})$ we have that \[ -\Lbod(K)^{\circ} = \Lbod(T(K)).\]	
\end{lem}

\begin{proof}
	Let $K = \{(x,s): s\ge \varphi(x)\}\in \C_{e_n,1}$, then by definition of $J$ we get that
	\[ J K= 
	\big\{ \Big(\frac{x}{s}, \frac{1}{s}\Big): s\ge \varphi(x)\big\} = \big\{ a\big(\frac{x}{\varphi(x)} , \frac{1}{\varphi(x)}\big) : 0\le a\le 1 \big\},
	\]
	and thus 
	\[ \Lbod(K) = \big\{ \big(a\frac{x}{\varphi(x)} , \pm a\frac{1}{\varphi(x)}\big) : 0 \le a\le 1 \big\}. 	
	\]
	Therefore
	\begin{eqnarray*}\Lbod(K)^{\circ} &=& \{ (y,r) : \iprod{x}{y} + rs \le 1 \quad\forall (x,s)\in \Lbod(K)\} \\ &=&
		\{ (y,r) : \iprod{ax}{y} \pm  r a \le \varphi(x) \quad\forall x \in \RR^{n-1}, 0\le a\le 1\}  
		\\ &=&
		\{ (y,r) : \iprod{x}{y} \pm  r  \le \varphi(x) \quad\forall x \in  \RR^{n-1}\}  
		\\ &=&
		\{ (y,r) : \iprod{x}{y}  - \varphi(x) \le   r  \le \varphi(x)-\iprod{x}{y} \quad\forall x \in \RR^{n-1}\}.  
	\end{eqnarray*}
	On the other hand, 
	\[ T(K) = \big\{ (x,s): s\ge \sup_z \frac{1 - \iprod{x}{z}}{\varphi(z)}   \} = \{ (x,s): s\varphi(z) \ge 1 - \iprod{x}{z} \quad \forall z \in \RR^{n-1}\big\}.\]
	So, 
	\begin{eqnarray*}JTK & =& \big\{ \big(\frac{x}{s}, \frac{1}{s}\big) :  s\varphi(z) \ge 1 - \iprod{x}{z} \quad \forall z \in  \RR^{n-1}
		\big\}  \\
		& = & \{ (y,r) : \frac{1}{r} \varphi(z)\ge 1 - \iprod{y/r}{z} \quad \forall z \in  \RR^{n-1}
		\} %
		\\
		& = & \{ (y,r) :  r \le   \varphi(z) +  \iprod{y}{z} \quad \forall z \in  \RR^{n-1}
		\} . 
	\end{eqnarray*}
	We thus see that
	\[    
	\Lbod(TK) = \{ (y,r) : -\varphi(z) -  \iprod{y}{z} \le  r \le   \varphi(z) +  \iprod{y}{z} \quad \forall z \in \ \RR^{n-1}
	\} . 
	\]
	Plugging in $-y$ instead of $y$ in $\Lbod(K)^\circ$, gives $\Lbod(TK)$ (since the bodies are invariant under reflection with respect to $e_{n}^\perp$, the map $r\mapsto -r$ does not change the set). 
\end{proof}  

Lemma \ref{lem:TisJ} essentially tells us that $\Lbod$ is a bijection between sets in $\C_{e_n,1}$ and convex bodies which include the segment $[-e_n,e_n]$, are included in the slab $\{ |\iprod{\cdot}{e_n} |\le 1\}$, and are invariant to reflections about $e_n^\perp$. Moreover, the mapping $T$ is the pull back of the mapping $L\mapsto -L^\circ$ on this class. 

In particular, the invariant sets for $T$ correspond to convex sets invariant under $L\mapsto -L^\circ$. The only centrally symmetric such set is the ball, but in general there are many sets invariant with respect to this operations, such as  the simplex (which is not in the class), a half-slab (which is in the class), and many more. 

\begin{rem}
	It is useful to note that the set $K_0 = \{ (x,y)\in \RR^{n}\times \RR^+: |x|^2 + 1 \le |y|^2\}$ corresponds to the ball under this transform. More precisely we have
	\[ \Lbod(K_0) = B_2^n.\]
	Indeed, \begin{eqnarray*} \{(x/s, 1/s): s\ge \sqrt{1+|x|^2}\} &=& \{ a(x/\sqrt{1+|x|^2}, {1}/{\sqrt{1+|x|^2}}): a\le 1\}\\
		&=& \{ (y,z): |y|^2 + z^2 \le 1, \,\,z\ge 0\}.
	\end{eqnarray*}
	
\end{rem}

\subsection{A Blaschke-Sanatal\'o type inequality}

Blaschke-Sanatal\'o type inequalities are a well investigated topic in convexity, but many intriguing open problems remain. The classical inequality  states that $\vol(K)\vol(K^\circ)\le \vol(B_2^n)^2$ for all centrally symmetric bodies $K$ in $\RR^n$, see e.g. \cite[Section 1.5.4]{AGAbook1}. For a non-centrally symmetric body, one needs to translate the body first (otherwise both the body and its dual can be huge). For example, the inequality is true when the origin is the bodies center of gravity, see \cite{meyer-pajor} or   \cite{AGAbook2}. Many variants of the Blaschke-Santal\'o inequality have been proven along the years, see, among others, \cite{FradeliziMeyer2007, Artstein2004santalo, lutwak1997blaschke, gao2003intrinsic},  where different size functionals and  duality operations are considered, or see \cite[Section 9.3.1-2]{AGAbook2}  and references therein. The reverse question, of bounding the volume product from below, is the famous Mahler's conjecture  \cite{Mahler},  for a discussion and   more references see \cite[Section 8.1]{AGAbook1}. The three dimensional centrally symmetric case of Mahler's conjecture was recently settled in \cite{IriyehShibata}. 
Some open problems of Blaschke-Sanatal\'o type include Cordero's conjecture \cite{Cordero}, which stipulates that $\vol(K\cap L)\vol(K^\circ\cap L)\le \vol(B_2^n\cap L)^2$
for centrally symmetric convex bodies $K, L$ (a version of this inequality where intersection is replaced by a new notion of 2-intersection was proved in \cite{klartag2007}).

In this subsection we will prove Theorem \ref{thm:newBStypeMAIN1}, which is a 
Blaschke-Sanatal\'o type inequality for the Gaussian measure and the new dual polarity. Central symmetry is replaced by a condition which we call ``essential symmetry''. 

\begin{definition}
	Let	$K\in \C\subset \P(\RR^n)$. We say $K$ is \emph{``essentially symmetric''} if for $u$ such that $K\in C_u$, identifying \ $\RR^n = u^\perp \times \RR u$, we have $(x,t)\in K$ implies $(-x, t)\in K$. Since $K\in \C_u$ can be interpreted as the epi-graph of a convex function   $\v:u^\perp \to \RR^+u$, essential symmetry is equivalent to $\v$ being an even function ($\v(-x) = \v(x)$ for $x\in u^\perp$). 	 
\end{definition}

We recall the statement of the main theorem we are proving in this section 

\begin{manualtheorem}{\ref{thm:newBStypeMAIN1}}
	Let $K\in\C \subset \P(\RR^n)$ be essentially symmetric.   Then 
	\[ \gamma_{n} (K)\gamma_{n}(TK) \le \gamma_{n}(K_0)^2,\]
	where $K_0 = \{ (x,y)\in \RR^{n-1}\times \RR^+: |x|^2 + 1 \le |y|^2\}$.   
\end{manualtheorem}

\begin{rem}
	Note that in this setting, one cannot simply shift a set in the class to be centered, however, since the Gaussian measure is a probability measure, it may be the case that the following theorem still holds without the essential symmetry assumption. 
\end{rem}

To prove Theorem \ref{thm:newBStypeMAIN1} we will use the above interpretation of $K$ as the $J$ image of some convex body, and an observation on the Jacobian of $F$, which was given in the paper \cite{artstein-florentin-segal}.

\begin{lem}
	Let $K \in \C_{e_n,1}\subset \P(\RR^n)$. Then 
	\[ \gamma_{n}(K)  = \nu(JK) \]
	where $d\nu(x,z) = (2\pi)^{-n} e^{-|x|^2/2z^2} e^{-1/2z^2}z^{-(n+1)}dxdz$ is defined on $\RR^{n-1}\times \RR^+$.
\end{lem}  

\begin{proof}
	The mapping  $F(x,z)  = (x/z, 1/z)$ (which is the point map corresponding to $J$) has a differential which is an upper triangular matrix with diagonal entries $(z^{-1}, \ldots z^{-1}, z^{-2})$. Therefore its Jacobian is $z^{-(n+1)}$. 
\end{proof}

Using Lemma \ref{lem:TisJ}, namely that polarity for $\Lbod K$ corresponds to $T$ for $K$ up to a minus sign,  Theorem \ref{thm:newBStypeMAIN1} will follow once we show the following theorem.

\begin{thm}\label{thm:nu-BS-symmetric}
	Let $L\subseteq \RR^{n}$ be a centrally symmetric convex body which includes the segment $[-e_n,e_n]$, is included in the slab $\{ |\iprod{\cdot}{e_n} |\le 1\}$, and is invariant to reflections about $e_n^\perp$. Then 
	\[ \nu(L) \nu (L^\circ) \le \nu (B_2^n)^2,
	\]
	where, as above,  $d\nu(x,z) = (2\pi)^{-n} e^{-|x|^2/2z^2} e^{-1/2z^2}z^{-(n+1)}dxdz$ on $\RR^{n-1}\times \RR^+$ (and with $d\nu(x,z)$ is $0$ for $z\le 0$).
\end{thm}

We will use the following inequality which was established in \cite{Cordero} for even dimensions, and as Cordero-Erausquin observes, holds in all dimensions by using Steiner symmetrizations as in Meyer and Pajor's  proof of the Santal\'o inequality \cite{meyer-pajor}.

\begin{lem}\label{lem:gaussian-BS}
	For a centrally symmetric convex set $L\subseteq \RR^n$ we have 
	\[ \gamma_n(L) \gamma_n(L^\circ) \le \gamma_n(B_2^n)^2.\] Moreover, for any $\alpha>0$ we have that 
	$\gamma_n(\alpha L) \gamma_n(\alpha L^\circ) \le \gamma_n(\alpha B_2^n)^2$. 
\end{lem}

\begin{proof}[Proof of Theorem \ref{thm:nu-BS-symmetric}]
	Given   $L$ as in the statement of the theorem, note that 
	\begin{eqnarray*} \nu (L) &=& (2\pi)^{-n} \int_0^1 e^{-1/2z^2}z^{-(n+1)} \left( \int_{ \{ x: (x,z) \in L\} } 
		e^{-|x|^2/2z^2} 
		dx\right) dz\\
		& = & (2\pi)^{-n} \int_0^1 e^{-1/2z^2}z^{-2} \left( \int_{ \{ y: (zy,z) \in L\} } 
		e^{-|y|^2/2}  
		dy\right) dz\\
		& = & (2\pi)^{-1} \int_0^1 e^{-1/2z^2}z^{-2} \gamma_{n-1} (\{ y: (zy,z) \in L\} ) dz.
	\end{eqnarray*}

	We define three functions on $[0,1]$: 
	\[ f(s) = e^{-1/2s^2}s^{-2} \gamma_{n-1} (\{ x: (sx,s) \in L\} ),
	\]
	\[ g(t) = e^{-1/2t^2}t^{-2} \gamma_{n-1} (\{ y:  (ty,t) \in L^\circ\} ),
	\]
	\[ h(r)  =
	e^{-1/2r^2}r^{-2} \gamma_{n-1} (\{ z: (rz,r) \in B_2^n\} ) =  e^{-1/2r^2}r^{-2} \gamma_{n-1} \big( \sqrt{\tfrac{1-r^2}{r^2}}\,B_2^{n-1} \big). 
	\]
	
	We claim that $f(s) g(t) \le h^2(\sqrt{st})$. If we can show this then we can use the multiplicative version of \PL \cite[Theorem 1.4.1]{AGAbook1} to get 
	\[ \int f \int g \le \left(\int h \right)^2,\]
	which translates to the statement of the theorem. 
	To show that the condition holds, we rewrite what needs to be shown as
	\[ \frac{e^{-1/2s^2}}{s^{2}} \gamma_{n-1} (\{ x: (x,s) \in L\} )
	\frac{e^{-1/2t^2}}{t^2} \gamma_{n-1} (\{ y:  (y,t) \in L^\circ\} )\le 
	\frac{ e^{-1/st}}{(st)^{2}} \gamma_{n-1}^2 ( \sqrt{1-st}B_2^{n-1}). 
	\]
	Note that since 
	$1/s^2 + 1/t^2 \ge 2/st $ we have 
	$e^{-1/2s^2 - 1/2t^2}\le e^{-1/st},$ and so the result will follow if we show that 
	\[  \gamma_{n-1} (\{ x: (sx,s) \in L\} )
	\gamma_{n-1} (\{ y:  (sy,t) \in L^\circ\} )\le 
	\gamma_{n-1}^2 \big( \sqrt{\frac{1-st}{st}}B_2^{n-1} \big). 
	\]
	Note that if $(sx,s) \in L$ and $(ty,t) \in L^\circ$ then we have that 
	$\iprod{x}{y} \le \frac{1 - st}{st}$. Hence, denoting ${\{ x: (sx,s) \in L\}  = : A}$  and 
	$\{ y:  (ty,t) \in L^\circ\}  = : B$ we see that, in particular, $$ \sqrt{\frac{st}{1-st}} B \subseteq\left( \sqrt{\frac{st}{1-st}} A \right)^\circ.$$

	From this inclusion, and with Lemma \ref{lem:gaussian-BS} in hand, we let $\alpha = \sqrt{\frac{1-st}{st}}$ to see that
	\begin{eqnarray*}  \gamma_{n-1} (A) 
		\gamma_{n-1}  (B)  &=&  \gamma_{n-1} (\alpha \sqrt{\tfrac{st}{1-st}}A) 
		\gamma_{n-1}  (\alpha \sqrt{\tfrac{st}{1-st}}B  )  \\&\le& \gamma_{n-1} (\alpha \sqrt{\tfrac{st}{1-st}}A ) \gamma_{n-1} \left(\alpha \big(\sqrt{\tfrac{st}{1-st}}A\big)^\circ \right) \\
		& \le & \gamma_{n-1} \big(\alpha B_2^n\big)^2 = \gamma_{n-1} \big(\sqrt{\tfrac{st}{1-st}} B_2^{n-1}\big)^2,
	\end{eqnarray*}
	which completes the proof. 
\end{proof}

\section{Invariant sets}\label{sec:fixed-points}

Finding the invariant sets of a given transform (also called ``fixed points'') is an interesting and often difficult task. Apart from the inherent interest of such questions, invariant sets can be important when analyzing extrema of certain functionals involving a set and its transform, as these extrema are in many cases (but not always) achieved on invariant sets (such is the case, for example, when the functional is symmetric with respect to taking the transform, and the extremum is unique - for instance in the classical  \BS inequality, see \cite[Subsection 1.5.4]{AGAbook1}). Another motivation to study invariant sets is that these enable us to study the possibility of two transforms being conjugate, see the paper \cite{higueras2022topological} 
for an instance where this method is used in the particular case of the polarity transform, and the general question of Anderson (for the Hilbert cube, which is quite similar to the space of convex bodies as is also illustrated there) on conjugate dualities is presented.

In this section we will study the existence and uniqueness of invariant subsets under \orqis. The two previously mentioned classical examples, Example \ref{ex:polar} and \ref{exm:Legendre}, have only one invariant set. The polarity transform in $\R^n$ (Example \ref{ex:polar}) has as its unique invariant set the Euclidean unit ball $B_2^n$. The \orqi associated with the Legendre transform (Example \ref{exm:Legendre}) has as a unique invariant set the epigraph of the function $\|x\|_2^2/2$. It is easy to check that some \orqis~ (such as the ``neighborhood complement'' in Example \ref{ex:complement-of-tneighborhood}) have no invariant sets, and some have multiple invariant sets (such as the %
``dual polarity'' discussed throughout Section \ref{sec:dual-duality}). 

We begin by stating and proving some general facts regarding invariant sets for an arbitrary order reversing quasi involution. 
These will capture many of the cases studied in the literature for well-known transforms. We analyze in detail some specific cases where more can be said, alluding to the examples previously discussed and other interesting \orqis. 

As above, we will denote an \orqi~ by $T:\P(X)\to \P(X)$. As we  established  in Theorem \ref{thm:duality-induced-by-cost} any \orqi is induced by some  symmetric  cost function $c:X\times X\to \RR$, and it is sometimes useful to use to notation   
$TK = K^c = \{ y: \forall x \in K, \, c(x,y) \ge 0\}$. Of particular importance in this section is the set 
\[ X_0 = \{ x: c(x,x) \ge 0 \}= \{ x: x\in T(\{x\})\},
\] due to the following simple fact.

\begin{lem}\label{lem:invsetinx0}
	Let $T:\P(X)\to \P(X)$ be an order reversing quasi involution. If $K = TK$ then $K\subseteq X_0$.
\end{lem}

\begin{proof}
	Let $x\in K$, then $\{x\}\subseteq K$ so we have $TK\subseteq T(\{x\}) $ and as $x \in TK = K$ we see that $x\in T(\{x\})$ so that $x\in X_0$. Since this holds for every $x\in K$ the claim follows.
\end{proof}

Clearly, if $X_0 = \emptyset$ (and $X$ is not the empty set, which is a standing assumption in this note) then there is no invariant set for $T$ because $T\emptyset = X$ (see Proposition \ref{prop:intersection-and-dual}). This is the case in Example \ref{ex:complement-of-tneighborhood}.

There are other instances of \orqis~with no invariant sets, for example one may consider a space $X$ with {\em any} \orqi $T$ and construct a space $X\cup\{z\}$ for some new point $z \not\in X$ such that $c(z,z)=-1$ and $c(x,z) = 1$ for all $x\neq z$. It terms of $T$, we define $\tilde{T}K= TK \cup\{z\}$ if $z\not\in K$ and $\tilde{T}K= T(K\setminus \{z\})$ if $z\in K$. In such a case, $\tilde{T}$ will have no invariant sets.

However, the theorem below shows that when $X_0 = X$ then
invariant sets always exist. In fact, for any $K_0$ with $TK_0 \supseteq K_0$ there exists an invariant set including  $K_0$. Moreover, when $X_0 \neq \emptyset$, omitting the condition $X_0 = X$ still there are always ``almost invariant sets'', namely sets for which $TK \cap X_0 = K$, and in some cases we can use this to show the existence of invariant sets, see 
Remark \ref{rem:AremarkonA}.

\begin{manualtheorem}{1.4}
	Let $T:\P(X)\to \P(X)$ be an \orqi and let $X_0$ as above. Let $K_0\subseteq X$ satisfy $K_0 \subseteq TK_0$. Then there exists some $K \subseteq X$ with $K_0 \subseteq K$ and such that $TK \cap X_0 = K$. In particular, if $X=X_0$ then for any $x_0\in X$ there exists  $K$  with $x_0 \in K =  TK$.
\end{manualtheorem}
\begin{proof}
	Consider the class ${\mathcal S}$ of subsets $S\subseteq X$ with $K_0\subseteq S$ and $TS\supseteq S$.  In terms of the cost $c$ this means $c(x,y) \ge 0$ for all $x,y\in S$.   The class $\mathcal S$ is non-empty, as it includes $K_0$. 
	Moreover, given a chain $S_\alpha$  (with respect to inclusion) in ${\mathcal S}$, consider the union of its elements $S_\infty = \cup S_\alpha$. Then $K_\infty$ still belongs to ${\mathcal S}$ as it will still include $K_0$ and for any $x,y\in S_\infty$ there is some $\alpha$ with 
	$x,y\in S_\alpha$ so that $c(x,y) \ge 0$ as well, implying $S_\infty \subseteq TS_\infty$.  
	Thus the conditions of Zorn's lemma are satisfied, that is, every chain has an upper bound. Therefore one may find a maximal element in ${\mathcal S}$. Denote it by $K$. Clearly $K\subseteq TK$ (since $K\in {\mathcal S}$), which means in particular $K\subseteq X_0$. We claim that $TK \cap X_0  = K$. Indeed, give $z\in TK \setminus K$ we know by maximality of $K$ that 
	$K_1 = K\cup \{z\}$ cannot belong to ${\mathcal S}$. In other words, $T(K_1) \not\supseteq K_1$. But as $c(x,z) \ge 0$ for all $x\in K$ (since $z\in TK$) and $c(x,y) \ge 0$ for all $x,y\in K$, this means that $c(z,z) <0$, that is, $z\not\in X_0$, as claimed.  Therefore, $K = TK\cap X_0$.  
\end{proof}

The only case where a multitude of invariant sets may occur, is when  $TX_0$ is a proper subset of $X_0$. 
In other cases, there is either no invariant set, or exactly one. More precisely we prove the following lemma.

\begin{lem}\label{lem:number-of-inv-sets}
	Let $T:\P(X)\to \P(X)$ be an \orqi and denote ${X_0 = \{x\in X: c(x,x)\ge 0\}}$,  as above.    
	\begin{enumerate}
		\item If $T X_0 = X_0$ then $X_0$ is the unique invariant set for the transform.
		\item If $ T X_0 \not\subseteq X_0$ then there is no invariant set for the transform.
		\item If $ T X_0\subsetneq X_0$ then there are examples where no invariant set exists,  examples  where only one invariant set exists, and examples  where more than one invariant set exists.  
	\end{enumerate} 
\end{lem}  

\begin{proof}
	For (1), if $T K = K$ then $K\subseteq X_0$ by Lemma \ref{lem:invsetinx0}, and  by order reversion we get that $K = TK  \supseteq T X_0 = X_0$ and so $K = X_0$.  
	
	To see that (2) holds,  by the same reasoning,  if $T K = K$ then we have that
	$ X_0 \supseteq K = TK  \supseteq T X_0$, and so if this inclusion does not hold, there will be no invariant sets.   
	
	For (3), let us give three simple examples: \textit{(i)} To give an example where no invariant set exists despite $X_0^c\subseteq X_0$, let $X$ 	be a four point space $X = \{1,2,3,4\}$.	
	Consider the cost given by $c(1,1) = c(2,2) = c(1,3) = c(3,1) = c(2,4) = c(4,2) = 1$, and  $-1$ for any other pair. 
	If there was some invariant set, it would have to be included in $X_0  = \{1,2\}$. However $\{1\} \mapsto \{1,3\}$ and $\{2\} \mapsto \{2,4\}$, whereas $X_0^c = \emptyset$ (and as usual $\emptyset \mapsto \{1,2,3,4\}$).

	\textit{(ii)}  An example with just one invariant set, despite $X_0^c\subsetneq X_0$, can be given with a three point space. Let $X = \{1,2,3\}$ and set $c(1,1) = c(2,2) = c(1,3) = c(3,1) = 1$ and all other pairs to have cost $-1$. In this case it is easy to check that $X_0 = \{1,2\}$ and $X_0^c = \emptyset$. The only invariant set is $\{2\}$, and the other subsets of $\{1,2\}$ satisfy $\{1\}\mapsto \{1,3\}$ and $\{1,2\}\mapsto \emptyset$ whereas $\emptyset \mapsto \{1,2,3\}$ of course. 
	
	\textit{(iii)}	Finally, fix some $X=\{x_i\}_{i\in I}$ with at least two elements. Let $c(x_i,x_i) = 1$ and for $x_i\neq x_j$ we let $c(x_i,x_j) = -1$. Then any $\{x_i\}$ is an invariant set, $X_0 = X$, and $X_0^c = \emptyset$. 
\end{proof}

While a complete characterization of invariant sets for general \orqis~ seems difficult, because of the various scenarios which can occur, the above facts capture many of the interesting examples. 
Our two favorite examples, classical polarity and Legendre transform, correspond to 
case (1) in Lemma \ref{lem:number-of-inv-sets}, where for polarity on $\RR^n$ we have   $X_0 = B_2^n$ and  for the set transform associated with Legendre we have $X_0 = \{ (x,t)\in \RR^n \times \RR : t\ge \|x\|^2/2\}$. 
In the case where $X_0 = \emptyset$, in which $TX_0 = X$, the lemma implies that there can be no invariant set
(unless $X = \emptyset$, a rather empty setting), but this was already clear from 
Lemma \ref{lem:invsetinx0} and Proposition \ref{prop:intersection-and-dual}. This is the case in Example \ref{ex:complement-of-tneighborhood}, for instance. 

\subsection*{Invariant sets for the flower duality} 

Another example with a unique invariant set corresponding to case (1) in the lemma is that of complements of flowers, see Example \ref{ex:flowers}. Recall that the associated set $S$ defining the transform is $S=\{(x,y): \sp{x,y}< \frac 12 |x|^2|y|^2\}$. In this case the invariant set is  ${X_0 = \RR^n \backslash \sqrt{2}B_2^n= \{x: \|x\|\ge \sqrt{2}\}}$. Indeed,  for $x,y\in X_0$ we have that
\[ \frac12 |x|^2|y|^2 \ge |x||y|\ge \iprod{x}{y},\]
where the first inequality holds since $|x|,|y|\ge \sqrt{2}$, and the second is due to Cauchy-Schwartz. This implies that $X_0 \subseteq T X_0$. For the other direction, assume $x\notin X_0$, hence $|x|< \sqrt{2}$. Choosing $y=\sqrt{2}\frac{x}{|x|}\in X_0$ we get that $\frac12 |x|^2|y|^2 =|x|^2$ and $\iprod{x}{y}=\sqrt{2}|x|$. Since the latter is clearly larger we get that $(x,y)\notin S$ which implies $x\notin TX_0$. So $TX_0=X_0$, and Lemma \ref{lem:number-of-inv-sets} item (1) tells us that this is the only invariant set.

\subsection*{Invariant sets for Rotem's transform} 
Let us discuss one other case of an \orqi with only one invariant set, namely  the \orqi we alluded to in 
Remark \ref{rem:Liran-duality}, which was considered by Rotem in   \cite{Liran-sharpBS}.

Rotem considered the functional transform $T_R$ given by 
\begin{equation*}\label{eq:rotem-transform}
 {{T_R}}\v (y) = \sup_{x\in \R^{n-1}} \frac{1+\sp{x,y}}{\v(x)},
\end{equation*}

which corresponds, as a set transform for epi-graphs in  $\RR^{n-1} \times \RR^+$, to a cost function $c((x,t), (y,s)) = ts-\iprod{x}{y} -1 \ge 0$. The set $X_0$, therefore,  consists of 
\[ \{(x,t):   t^2 \ge  |x|^2 +1  \}. \]
It is easy to check  that $v(x) = \sqrt{|x|^2 + 1}$ is an invariant function, and we once again see an instance with exactly one invariant set.

\subsection*{Invariant sets for the dual polarity} 
We turn our attention to the new duality transform presented in Example \ref{ex:dual-polarity} and discussed in Section \ref{sec:dual-duality}. 

While at first seemingly similar to Rotem's transform, the situation is completely different when we use the dual-polarity transform as in Example \ref{ex:dual-polarity}, or, for an even simpler case, the modified transform where in $T_R$ we add a negative sign, getting the transform in Proposition \ref{prop:functional-form-or-dp}.
Let \[ {{T_N}}\v (y) = \sup_{x\in \R^{n-1}} \frac{1-\sp{x,y}}{\v(x)},\]
which corresponds, as a set transform for epi-graphs in  $\RR^{n-1} \times \RR^+$, to a cost function $c((x,t), (y,s)) = \iprod{x}{y}+ts -1 \ge 0$. The set $X_0$, therefore,  consists of 
\[ \{(x,t):   t^2 \ge   1 - |x|^2  \}, \]
which is not in the associated class an in particular is not an in variant set. 
In other words, the ``dual polarity'' transform, to which Section \ref{sec:dual-duality} was devoted, is an \orqi corresponding to case (3) in Lemma \ref{lem:number-of-inv-sets} that has many invariant sets. 

Indeed, here are a few different invariant sets of the above transform: 

\[K_0 =  \{(x,t):   t^2 \ge  |x|^2 +1  \}, \]
\[K_1 =  \{(x,t):   x\ge 0, t\ge  1  \}, \]
and in $\RR^2$, for example, 
\[ K_2 =  \{ (x,t) : x\ge 0 {\rm ~and~} t\ge x + 1{\rm~or~}-1\le x\le 0{\rm ~and~}t \ge 1 {\rm~or~}  x\le -1{\rm ~and~}t \ge -x\}.            \]
In fact, in $\RR^2$ we may give the general form of an invariant set for this transform, as the following Lemma explains. 

\begin{prop}
	Consider the cost $c((x,t), (y,s)) = xy+ts-1$ on $\RR^2$, and the associated transform $T: A\mapsto A^c$. 
	Let $K\subset \RR^2_+$ such that $e_2 =(0,1) \in K$ is its closest point to the origin, i.e. $K\in \S=\C_{e_2,1}$. Define  $$L=(K^c\cap \{(x,t)\in \RR^2: x<0, \ t\ge 0\} ) \cup (K^{cc} \cap \RR^2_+).$$ 
	Then the set $L$ belongs to $\S$ and it is an invariant set for $T$. Moreover, any invariant set for $T$ which belongs to $\C_{e_2}$  is of this form. (In particular, this characterizes the invariant sets with respect to $T_N$.)
\end{prop}

\begin{proof}
Let $L$ be  of the form given in the statement of the proposition. It follows from Lemma \ref{lem:invariant_subclass}, and the discussion that follows it, that $L\in \S$.

To show the invariance, we first prove that $L\subseteq L^c$. Fix some $(x,t)\in L$. If $x\ge 0$ this means $(x,t)\in K^{cc}$. Then for any $(y,s) \in L$ either $y<0$ in which case $(y,s) \in K^c$ and so $xy+st \ge 1$, or $y\ge 0$ in which case, as $x,y\ge 0$ and $t,s\ge 1$, we again get $xy+ts\ge ts \ge 1$. If $x<0$ a similar argument works. This shows $L\subseteq L^c$.

For the opposite inclusion, namely that $L^c\subseteq L$, let $(x,t)\in L^c$. Assume that $x\ge 0$ and consider a point $(y,s)\in K^c$. 
Then, either $y<0$ from which it follows that $(y,s)\in  L$ and hence $xy+st\ge 1$ by the definition of the transform, or $y\ge 0$ in which case as $x\ge 0$ and $L^c,K^c\in \S$ we get that $xy+st\ge st\ge 1$. Therefore, we conclude that $(x,t)\in K^cc$, and as we assumed $x\ge 0$, we get that $(x,t)\in L$.
A similar argument shows that any point $(x,t)\in L^c$ with $x< 0$ belongs to $K^c$ and thus to $L$.
 
We conclude that all sets of the form given above are invariant. To show that all invariant sets within the sub-class $\C_{e_2}$ must be of this form, note that they must in fact belong to $\C_{e_2,1}\subset \C_{e_2}$, and that any set satisfying $K=K^c$ must be in the sub-class, i.e. $K=K^{cc}$. Finally, since the sets in $\C_{e_2,1}$ are contained in the upper half plane, we get that ${K=(K^c\cap \{(x,t)\in \RR^2: x<0, \ t\ge 0\} ) \cup (K^{cc} \cap \RR^2_+)}$.
\end{proof}

In view of Lemma \ref{lem:TisJ}, invariant sets for the dual polarity correspond, after application of $\tilde{J}$, to invariant sets for the transform $K\mapsto -K^\circ$, which also have the property that they are included in the slab $\{ (x,t): |t|\le 1\}$ and include the segment $\{(0,t): -1\le t\le 1\}$. Here again, in $\RR^2$ this characterization is quite simple, and in higher dimensions a clean way of expressing such sets is yet to be given.

One other interesting \orqi corresponding to case (3) of Lemma \ref{lem:number-of-inv-sets} whose invariant sets remain quite mysterious is the one associated with the functional transform $\A$, whose importance was established in \cite{before-hidden,ArtsteinMilmanHidden}. %
We delay the discussion of the transform and its invariant sets to the next section (see Remark \ref{rem:AremarkonA}), and only mention here the fact that many self-dual functions exist, as exhibited in \cite{rotem2012characterization}, and that using the methods developed above we can use specially designed subsets $K_0$ which, using Theorem \ref{thm:inv-sets-with-intersection}, allow us to conclude the existence of many invariant sets.

\section{Another set of examples}\label{sec:another-ex}

In this section we gather, in no particular order, more examples that are instructive for the discussion about \orqis. 
We start with examples of a functional nature, similar to the one induced by the Legendre transform presented in Example \ref{exm:Legendre}. 
When considering an \orqi on a class of sets, if these sets are epi-graphs of functions we can view the transform on them as a functional transform. The inclusion order on epigraphs translates to a pointwise inequality of the functions. We say that $T$ is an \orqi on a class $\C$ of functions $f:X\to [-\infty, \infty]$  if for any $f,g\in \C$ we have
\begin{enumerate}[(i)]
	\item $TT f \le  f$ \tabto{5.5cm} (quasi involution)
	\item $f \le g$ implies $T f \ge T g$ \tabto{5.5cm} (order reversion)
\end{enumerate}

In other words, a functional duality $T$ corresponds to an \orqi on $\P(X\times \RR)$, which we also denote by $T$, whose image class is the set of epi-graphs of functions in the image of $T$. 
We define it on sets which are epi-graphs by 
$T(\epi(f)) := \epi(Tf)$.

\begin{exm}[Functional Polarity]\label{exm:A-transf}
	Consider the class $\cvx_0(\RR^n)$ of geometric convex functions on $\RR^n$ (non-negative lower semi-continuous convex functions with minimum 0 at the origin)  and on it the order reversing involution   $\A$, given by
	\[\A\varphi(y)=\sup_{y\in \RR^n}\frac{\iprod{x}{y}-1}{\varphi(y)}.\]
	For discussion of this transform and its many properties, see \cite{ArtsteinMilmanHidden}. The family of epi-graphs of functions in $\cvx_0(\RR^n)$ is closed under intersection. Further, the transform on it can be geometrically described as a composition of standard duality in $\RR^{n+1}$ with a reflection (see again \cite{ArtsteinMilmanHidden}). Therefore, the transform is induced by the following cost 
	\[c((x,t),(y,s)) = st+ 1 -\iprod{x}{y}.  \] 
\end{exm}

\begin{rem}\label{rem:LandAon2hom}
	It is well known that $\L$ and $\A$ agree on the sub-class of $\cvx_0(\RR^n)$ consisting of 2-homogeneous  functions. 
	However, on the whole $\cvx_0(\RR^n)$ the transforms $\L$ and $\A$ differ substantially, and therefore so do their extensions to $\P(\RR^{n+1})$. We thus have an example of an \orqi (on 2-homogeneous convex functions in $\RR^n$) which admits two different extensions (See also Subsection \ref{sec: extension}). Of course, each of these extensions has a different image ($\cvx(\RR^n)$ and $\cvx_0(\RR^n)$, respectively). 
\end{rem}

\begin{rem}[{\bf Fixed points for $\A$}]\label{rem:AremarkonA}
	Let us analyze this important example with regard to its multiple invariant sets.  The rotation-invariant invariant sets for this transform were analyzed by Rotem in \cite{rotem2012characterization}. Let us start with the case of functions on the ray $\RR^+$. In this example,     a point in the epigraph can be written as $\bar{x}  = (x,t)$ with $x\in \RR^+$ and $t\ge 0$.
	The cost is given by $c((x,t), (y,s)) = st+1 -  {x}{y}$  and so $X_0$ is the set $(x,t)$ with $t^2 + 1 - x^2 \ge 0$, or $t\ge \sqrt{x^2 -1}$. 
	If we fix a point $(x_0,t_0)$ on the boundary of $X_0$, then  for points outside $X_0$, say $(y,s)$ with $s^2 < y^2-1$, we have (as clearly $x,y\ge 1$) that
	\[c((x,t), (y,s)) = st+1 -  {x}{y} 
	= s\sqrt{x^2 -1} +1 -  {x}{y} < \sqrt{y^2 -1}\sqrt{x^2 -1} +1 -  {x}{y}\le 0.
	\] 
	In other words, having chosen some $\bar{x}$ on the boundary of $X_0$, and taken $K$ to be an invariant set which includes $\bar{x}$ (which exists by Theorem \ref{thm:inv-sets-with-intersection}), we get that $K^c\subseteq X_0$ automatically and so $K^c = K$. In other words, for this specific example, for any point in the boundary of $X_0$, there exists an invariant set which includes this point. It is not unique, however, see 
	\cite{rotem2012characterization} for a full characterization. 
	
	Moreover,  no two different points on the boundary of $X_0$ can belong to the same invariant set (as the inequality above remains valid, where the strictness of the inequality follows from strictness of the arithmetic geometric inequality when the points are different).  This is a reason for the multitude of invariant sets in this specific case.  %
	
	In higher dimensions a similar argument can be used to construct various non-rotation-invariant sets. For example, consider the set in $\RR^3$ given by  
	\[
	K_0 = \{ (1,y,|y|): y\in \RR\}\cup \{   (-1,y,|y|): y\in \RR
	\}.
	\]
	Note that for any pair of points in this set we have
	\begin{eqnarray*}
		c((1,y,|y| ) , (1,z,|z|)  )  &=& |y||z| + 1 - \iprod{(1,y)}{(1,z)} = |y||z| - \iprod{y}{z} \ge 0\\
		c((1,y,|y| ) , (-1,z,|z|)  )  &=&  |y||z| + 1 - \iprod{(1,y)}{(-1,z)} = 2+ |y||z| - \iprod{y}{z} \ge 0
	\end{eqnarray*}
	which means $K_0 \subseteq TK_0$. 
	Theorem \ref{thm:inv-sets-with-intersection} then implies that there exists some ``almost invariant set'' including it, that is, some $K = TK\cap X_0$ such that $K_0 \subset K$. However, similar to the two dimensional example above, the set $K_0$ includes specially designed points on the boundary of $X_0$ which in fact imply that $TK = K$.
	
	In fact, in this specific example, it is not hard to check that $K = K_1 := \conv(K_0)$ and this set   satisfies $TK_1= K_1$. Indeed, since sets in the class must be convex, any $K$ which includes $K_0$ must include $K_1$. The fact that $TK_1 = K_1$ can be checked manually: $K_1\subset TK_1$ by the same argument as before, since if $x,a,\in (-1,1)$,  $z\ge |y|$ and $d\ge |b|$ then 
		\begin{eqnarray*}
		c((x,y,z ) , (a,b,d)  )  &\ge & |z||d| + 1 - \iprod{(x,y)}{(a,b)} = |z||d| +1 - xa - \iprod{y}{b} \ge 0.
	\end{eqnarray*}
	
	Next take $(a,b,d)\notin K_1$ and apply a simple case analysis for finding $(x,y,z)\in K_1$ such that $c((x,y,z),(a,b,d))< 0$: If $a>1$   take $(x,y,z)=(1,0,0)$ and if $a<-1$ take $(x,y,z)=(-1,0,0)$. If $-1\le a \le 1$ then by assumption and $d<|b|$. Write $d=|b|-\varepsilon$ for some $\varepsilon>0$, and let $x=0$, $y={\rm sign}(b)z$ and $z>1/\varepsilon$.  Then $(x,y,z)\in K_1$  
	since it is the convex combination of $(-1,y,|y|)$ and $(1,y,|y|)$. However, $c((x,y,z),(a,b,d)) = (|b|-\varepsilon)z +1 - |b|z  <0$. We conclude that no point outside $K_1$ can be in $TK_1$, so that $TK_1 = K_1$. 
\end{rem}

The  example of functional polarity $\A$, as well as $\L$, both correspond (up to composition with a monotone one-dimension function) to a 
special class of functional  \orqis~ which are used in optimal transport theory. These are the so-called cost-transforms (see \cite{villani-book}), defined as follows.

\begin{definition}\label{def:func-c-trans}
	Let $c:X\times X \to (-\infty,\infty]$ be some cost function. For $\psi:X\to [-\infty, \infty]$  its {\bfseries $c$-transform} is given by
	\begin{equation*}\label{eq:c-transform} 
	\psi^c(x)=\inf_{y\in X} (c(x,y)-\psi(y)). \end{equation*}
\end{definition}

\begin{exm}[Hypo-graphs]\label{ex:functional-cost-trans}  
	In general, a cost transform
	induces a transform on hypo-graphs of functions, namely on sets $K\subseteq \RR^{n+1}$ such that for any fixed $x\in \RR^n$ the set {$\{t: (x,t)\in K\}$} is of the form $(-\infty, a]$ for some $a\in [-\infty, \infty]$, which is a class  closed under intersection.  Call $\{(x,t): t\le \psi(x)\} = \hypo(\psi)$. Given $T: \C\to \C$ where $\C$ is a class of functions $\psi: \RR^n \to [-\infty, \infty]$, which satisfies $TT\psi \ge \psi$ and $\psi \le \varphi$ implies $T\psi \ge T\varphi$ we let 
	\[ T(\hypo (\psi)) := \hypo (T\psi). \] 
	For example, $\psi \mapsto \psi^c$ satisfies these properties (see, e.g., \cite{paper1}). 
	
	As before, since $\psi^{cc} \ge \psi$ we have that $\hypo (\psi^{cc}) \supseteq \hypo(\psi)$ so that $T$ is an \orqi on the class of hypo-graphs of functions in $\cvx_0(\RR^n)$.  
	
	It is easily checked that the cost function corresponding  to the \orqi on sets is given by 
	\[ c((x,t), (y,s)) = c(x,y)-s-t.\] 
	
\end{exm}

We remark that there exist much more sophisticated functional quasi involutions in various settings, for instance, see \cite{berndtsson2019complex} for various Legendre type transforms on complex functions.

Going back to some 
trivial examples,  we consider \orqis~ whose image consists of a finite number of sets. In this case, determining the various \orqis~ reduces to a question of a combinatorial nature. %

\begin{exm}[Simple \orqis]
	\begin{enumerate} 
		\item When there is a single element in the class, it must be all of $X$ and then the \orqi~  is such that all sets are mapped to $X$.
		\item  
		When the image includes two sets they must be $K\neq X$ and $X$, and the only possible duality takes all subsets of $K$ to $X$, and anything which is not a subset of $K$ to $K$.
		\item %
		An \orqi~ whose image has three elements must have as these three elements the set $X$, and two comparable sets $K_1\subseteq K_2$, where $T$ switches $X$ and $K_1$, while $K_2$ is a fixed point for $T$. The pre-image of $X$ is precisely all $L$ such that $L \subseteq K_1$. 
		The pre-image of $K_2$ consists of subsets of $K_2$ which are not subsets of $K_1$. All other sets are mapped to $K_1$. 
		\item 	A similarly trivial example is when ${\mathcal C} = \{ X, K_0, K_1, K_2\}$ where $K_0 \subseteq K_1\subseteq K_2$, and the mapping switches the two, namely $TK_1 = K_2$ and $TK_2 = K_1$.  
		If $K\subseteq K_0$ then $T(K) = X$, if $K \subseteq K_1$ and $K\not\subseteq K_0$ then $ {T}(K) = K_2$, if  $K  \subseteq K_2$ and $K\not \subseteq K_1$ then $TK = K_1$ and if $K \not\subseteq K_2$ then $TX = K_0$.  
	\end{enumerate}
\end{exm}

One can of course continue in this manner, a path which we will not pursue here.

Moving on, we discuss \orqis~corresponding to some well known classes of bodies. 

\begin{exm}[Unconditional bodies]
	Let $X=\RR^n$ and consider the set 
	\[ S=\{(x,y): \sum_{i=1}^n |x_i||y_i|\le 1\}.\] 
	We claim that the associated class is the class of unconditional convex bodies and that on it, the  associated transform is the usual polarity. Moreover, for a general set, applying this transform twice produces the smallest unconditional body which contains the set (which we call its unconditional convex hull). Finally, the only invariant set is the Euclidean unit ball. 
	
	To show the first claim, note that if $x= (x_i)_{i=1}^n\in K^c$ then also $(\pm x_i)_{i=1}^n\in K^c$ since the cost depends only on the absolute values of the coordinates. In particular, we see that  the  class consists only of unconditional bodies. The fact that on unconditional bodies the transform induces the usual polarity is immediate since $c(x,y) = 1-\sup_{\pm} \iprod{(\pm x_i)_{i=1}^n}{y}$. It follows that all unconditional convex bodies are in the image, therefore in the $c$-class. 
	
	By Proposition \ref{prop:TTKisenvelope} we get that $TTK$ is the smallest unconditional body which includes $K$, and $TK$ is its polar. 
\end{exm}

Our next example has to do with 
the space $X=\RR^n$ and the subset $\C\subseteq \P(X)$ of closed star shaped sets, namely those sets $A\subseteq \R^n$ which satisfy $A=\{\lambda x: x\in A\,\  0\le \lambda\le 1\}$. Given a closed star shaped set $A$, denote by $g_A$ its gauge function which uniquely determines it. We discuss cost dualities whose image is this class. As it turns out, there are various different dualities on the class.

\begin{exm}[Star-shaped]\label{example:star-shaped1}
	Consider the cost 
	\[c(x,y) =\begin{cases}
	-\iprod{x}{y}+1 & y\in \R_+x\\
	+\infty & \text{otherwise}
	\end{cases}.\]
	The image of the $c$-transform is the class of star shaped bodies, and the $c$-duality on it takes the gauge function $g$ to $1/g$.
	
	Indeed, consider a closed star shaped $A$. Then 
	\begin{eqnarray*}
		A^c &=& \{ y\in \RR^n: \forall x\in A,  c(x,y) \ge 0 \}\\
		& = &  \{ y\in \RR^n: \forall x\in [0, y/g_A(y)],  \iprod{x}{y} \le 1 \}\\
		& = &  \{ y\in \RR^n:   \iprod{y/g_A(y)}{y} \le 1 \}\\
		& = &  \{ y\in \RR^n:   |y|^2\le g_A(y)  \}.
	\end{eqnarray*}
	This set is clearly closed, star shaped, and in fact  for $u\in S^{n-1}$, if  $g_A(u) = \lambda$ then
	\[g_{A^c}(u) = \inf \{ r: u/r\in A^c\} = \inf \{ r: |u/r|^2 \le 
	g_A(u/r) \} = \inf \{ r: 1/r  \le 
	\lambda  \}  = 1/\lambda=1/g_A(u). 
	\] 
	In particular, $A\mapsto A^c$ is an involution on the class of closed star shaped domains (not necessarily bounded, and $0$ can be on the boundary). It is also easy to check that $A^c$ is always closed and star shaped.
	
\end{exm}

\begin{rem}
	We may generalize the cost from Example \ref{example:star-shaped1} to get a large family of dualities on star shaped domains as follows. For any $u\in S^{n-1}$ let $f_u: [0, \infty)\to [0, \infty)$ be an increasing bijective function, such that $f_u$ depends continuously and symmetrically on $u$. Let $F:S^{n-1}\to S^{n-1}$ be continuous and satisfy $F\circ F = Id$. Then
	\[c(x,y) =\begin{cases}
	-f_{y/|y|}(\iprod{x}{|y|F(y/|y|)})+1 & y/|y|=F(x/|x|)\\
	+\infty & \text{otherwise}
	\end{cases}\]
	is a cost whose image is, again, closed star shaped sets. 
\end{rem}

\section{Some additional useful facts}\label{sec:some additional}
In this section we gather many more  general facts about \orqis, which we did not want to include in the previous sections so as not to overcrowd them. These include the investigation of necessary and sufficient conditions for the possibility to extend an \orqi defined on a subset of $\P(X)$,  the possibility of composition of maps to induce new \orqis, and the restriction of \orqis~to be defined on subsets of some $Y\subseteq X$.

\subsection{Extensions of order reversing quasi involutions}\label{sec: extension}
In this subsection we will study when and how an order reversing quasi involution $T$ defined on a subfamily $\C\subseteq \P(X)$ can be extended to the whole of $\P(X)$. This serves to  explain the somewhat surprising  phenomenon that while, say, polarity for sets, is naturally defined on convex subsets which include the origin, it can be defined on a general subset of $\RR^n$. Similarly, while the natural domain for Legendre transform is $\cvx(\RR^n)$, one may define the Legendre transform of any function. This turns out to be a feature of \orqis~with a certain additional condition. 

We begin by showing uniqueness of the extension,  namely that an order reversing quasi involution on  $\C\subseteq \P(X)$ cannot have more than one extension to $\P(X)$ if the image is to remain $T(\C)$. 

\begin{prop}\label{prop:image determines all} 
	Let $T_1, T_2: \P(X) \to \P(X)$ be \orqis, and assume $T_1(\P(X)) = T_2(\P(X)) =: {\mathcal C}_0\subseteq \P(X)$, and further assume $T_1|_{{\C}_0} = T_2|_{{\C}_0}$. Then $T_1 = T_2$. 
\end{prop}

\begin{proof}
	We let $T$ stand for either $T_1$ or $T_2$ and show that $TL = T\left(\cap \{K\in {\mathcal C}_0: L \subseteq K\}\right)$, which in particular, implies that $T_1$ and $T_2$ are equal.
	Note that ${\mathcal C}_0$ is closed under intersections, and includes $X$ as $TTX = X$. Moreover, as $T$ is an involution on ${\mathcal C}_0$, we know that letting $K_0 = \cap_{K\in {\mathcal C}_0} K$ we have $TK_0 = X$ and $TX = K_0$.  
	Clearly all subsets of $K_0$ are mapped to $X$ as well.

	Let $L\in P(X)\setminus {\mathcal C}_0$, and let $L_0 = \cap \{K\in {\mathcal C}_0: L \subseteq K\}$. The intersection is non-empty as $L\subseteq X$, and by the fact that $\C_0$ is closed under intersection,  $L_0\in \mathcal C_0$. By order reversion, as $L\subseteq L_0$, we know $TL\supseteq TL_0$ (and $TTL \subseteq TTL_0 = L_0$). 
	On the other hand, as $L \subseteq TTL$ and $TTL \in {\mathcal C}_0$, we know that $TTL$ participates in the intersection defining $L_0$, namely $L_0 \subseteq TTL$. By order reversion, and the involution property of $T$ on ${\mathcal C}_0$, we get $TL_0 \supseteq TTTL = TL$. That is, $TL = TL_0$. 
\end{proof}

\begin{rem}
	Within the proof  we have shown the useful fact that for an order reversing quasi involution with image ${\mathcal C}_0$, 
	\[  TL = T\left(\cap \{K\in {\mathcal C}_0: L \subseteq K\}\right).\] 
\end{rem}

We proceed to showing that given an order reversing quasi involution on a subset ${\mathcal C}$ of $\P(X)$, which  is closed under intersection, and with  $X\in {\mathcal C}$, we may extend it to the full space $\P(X)$. 

\begin{lem}\label{lem:extending}
	Let ${\mathcal C}\subseteq \P(X)$ be a class of sets which is closed under intersection, namely if $(K_i)_{i\in I} \subseteq \C$ then $\cap_{i\in I} K_i \in \C$, with $X\in \C$.  Given an order reversing quasi involution  $T:{\mathcal C}\to {\mathcal C}$, there exists  an order reversing quasi involution $\hat{T} :\P(X) \to \P(X)$ with $\hat{T}|_{\mathcal C} = T$, and $ {\rm Im}(\hat{T})=\C$. 
\end{lem}

\begin{proof}%
	Let $\C_0 = \{TK : K \in \C\}$ and let $L_0\in \C$ denote the intersection of all elements in $\C_0$ (it may be an empty set). Since $X\in \C$, and since for all $K\in \C$, $K\subseteq X$, we know that $TX \subseteq TK$ for all $K\in \C$ so  that $TX\subseteq \cap_{L\in \C_0} L = L_0$, and as $TX$ itself participates in the intersection, we see $TX = L_0$. 
	
	For $K\subseteq X$ we define 
	\[ \hat{T}(K) = T (\cap \{L\in \C: K\subseteq L\}). \] 
	First note that as $\C$ is closed under intersection, the set on which we apply $T$ (which we call the ``$\C$-hull'' of $K$) is in $\C$ as well.  We take into account here the fact that $X\in \C$, so that the intersection is over a non-empty set of elements $L$.   	
	Therefore, $T$ is well-defined. 
	Next, note that for $K\in \C$, this  ``$\C$-hull''  is $K$ itself, and therefore $\hat{T}|_\C = T$. Let us show that $\hat{T}$ is an order reversing quasi involution. For the order reversion, let $K_1\subseteq K_2$, in which case the ``$\C$-hulls'' satisfy the same inclusion $\cap \{L\in \C: K_1\subseteq L\}\subseteq \cap \{L\in \C: K_2\subseteq L\}$, and by order reversion of $T$ on $\C$ we get  
	\[ \hat{T}(K_1) = T (\cap \{L\in \C: K_1\subseteq L\})  \supseteq  T (\cap \{L\in \C: K_2\subseteq L\}) = \hat{T} (K_2).\] 
	For the quasi involution note that a set is always included in its ``$\C$-hull'', $K\subseteq \cap \{L\in \C: K\subseteq L\}$, and since $T$ is a quasi involution we get 
	\[  K\subseteq \cap \{L\in \C: K\subseteq L\} \subseteq TT(\cap \{L\in \C: K\subseteq L\}) =  \hat{T}T(\cap \{L\in \C: K\subseteq L\})  = \hat{T}\hat{T}(K).\]
	where we used (for the second equality from the right) that  $T = \hat{T}$ on $\C$, and the image of $T$ is in $\C$. Finally, due to the fact that $\hat{T}(K)\in \C$ by the definition of $\hat{T}$, we see that ${\rm Im}(\hat{T})=\C$.
\end{proof}

As a corollary we get that the image of an order reversing quasi involution on $\P(X)$ is always closed under intersection (this also follows directly from Proposition \ref{prop:intersection-and-dual}).

Some important classes of sets $\C$ to which one can apply Lemma \ref{lem:extending} are closed sets,  convex sets,  sets which include the origin, bounded sets (together with the whole space), sets which are epi-graphs (resp. hypo-graph)  of functions with values in $(-\infty, \infty]$ (resp. $[-\infty, \infty)$),   any intersection of the above, and many more.

Joining Proposition \ref{prop:image determines all} and Lemma \ref{lem:extending} we get the following corollary.
\begin{cor}
	Let $\C\subseteq \P(X)$ a family of sets which is closed under intersections, and let $T:\C\to \C$ be an \orqi on $\C$ with ${\rm Im}(T)=\C$. Then, $T$ can be \emph{uniquely} extended to an \orqi $\hat{T}:\P (X)\to \P(X)$, such that $\hat{T}|_\C =T$ and ${\rm Im}(\hat{T})=\C$.
\end{cor}

Recall that as the example of Legendre transform $\L$ and polarity transform $\A$ on, say, 2-homogeneous convex functions shows,  an \orqi may have more than one extension to the whole space if we do not insist on keeping the same image. For details see Remark \ref{rem:LandAon2hom}.

It turns out that if one is given an \orqi defined on a subset $\C\subseteq \P(X)$ which is not closed under intersection, it may be the case that there is no extension to an \orqi on $\P(X)$, as is the case in the following example. 

\begin{exm}\label{exm:no-extension}
	Consider the following sets $A=\{1,2\}, \, B=\{1,3\}, \, C=\{2,3\}$, as subsets of $X=\{1,2,3\}$.  Note that no two of these sets are comparable with respect to inclusion and hence any involution on these three sets will be order reversing (in an empty sense).  Set $TA = B$, $TB = A$ and $TC = C$.  We then have that $A\subseteq B\cup C$ but $TA = B\not\supseteq A \cap C = TB \cap TC$. 
	However, any extension $\hat{T}$ of $T$ would have to satisfy (by Proposition \ref{prop:intersection-and-dual}) that 
	\[  \{1,3\}  = T(A)  = \hat{T} (A) \supseteq \hat{T} (B \cup C) = \hat{T}(B) \cap \hat{T}(C) = TB \cap TC = \{2\} \]
	which is not true, therefore, an extension $\hat{T}$ does not exist. 
\end{exm}

The reason that this example fails to have an extension is quite simple. For an \orqi on $\C\subseteq \P(X)$ to admit an extension to $\P(X)$,  Proposition \ref{prop:intersection-and-dual} implies that it must satisfy the following property: 

\begin{definition}\label{def:respects-inclusions}
	Let $\C\subseteq \P(X)$ and let $T:\C\to \C$ be an order reversing quasi involution. We say that $T$ \textbf{respects inclusions} if
	for any $K,K_i \in \C$, $i\in I$ such that $K\subseteq  \cup_{i\in I} K_i$ one has that $TK \supseteq\cap_{i\in I} TK_i $.
\end{definition}

Note that by Lemma \ref{lem:extending} we know that if an \orqi $T$ is defined on a family $\C$ which is closed under intersections then it can be extended to the whole space, and in particular, Proposition \ref{prop:intersection-and-dual} implies that $T$ respects inclusions. 

It turns out that the condition of respecting inclusions is not only necessary, but sufficient for the existence of an extension, as we show next.

\begin{manualtheorem}{1.6}\label{prop:extending-general-class}
	Let $\C\subseteq \P(X)$ be a family of sets and $T:\C\to \C$ be an \orqi which respects inclusions. Then \, $T$ can be extended to an \orqi $\hat{T}:\P (X)\to \P(X)$, such that $\hat{T}|_\C =T$. 
\end{manualtheorem}  

\begin{proof}
	Given the transform $T$ define $S = \cup \{   TK \times TTK: K \subseteq X\}\subseteq X\times X$ as in Remark \ref{rem: set S}.  We may then define a cost function $c$ such that $c(x,y)=1$ if $(x,y)\in S$ and $-1$ otherwise (this is the same as in the proof of Theorem \ref{thm:duality-induced-by-cost}).  Clearly, the corresponding cost transform is defined on the whole of $\P(X)$ and we must only check that for any $K\in \C$ we have that $TK=K^c$. 
	
	To show that $TK\subseteq K^c$ note that by definition of $c$, for any $(x,y)\in TT K\times T K$ we have $c(x,y)=1$. In particular $c(x,y)=1$ for all $(x,y)\in K\times TK$, as $K\subseteq TTK$ by property (i). Hence,	\[K^c = \{y:\ \forall x\in K\ c(x,y)\ge 0\} \supseteq \{y:\  (x,y)\in K\times TK\}=TK\]	which gives $TK\subseteq K^c$.
	
	For the other direction, note that
	\begin{align*}
	K^c&=\{y: \forall x\in K \, c(x,y)\ge 0\} = \{y:\forall x\in K \,(x,y)\in S \} \\
	&= \{ y: \, \forall x\in K \, \exists L_{(x,y)}\in \C \text{ with } (x,y)\in L_{(x,y)}\times TL_{(x,y)}\}.
	\end{align*} 
	Fix some $y_0\in K^c$,  and note that the above implies that we must have $K\subseteq \cup_{x\in K} L_{(x,y_0)}$ and at the same time, since $y_0 \in TL_{(x,y_0)}$ by construction, we know that $y_0 \in \cap_{x\in K}  TL_{(x,y_0)}$. By the assumption that $T$ respects inclusions we conclude that $TK \supseteq \cap_{x\in K}  TL_{(x,y_0)}$ and hence $y_0\in TK$. Since $y_0 \in K^c$ was arbitrary, it follows that $K^c \subseteq TK$ as required. 
\end{proof}

\subsection{Composition and conjugation}\label{subsec: Com and Con}

There is a close relation between order reversion and order preservation.  We collect a few facts.

\begin{lem}
	Let $S, T: \P(X)\to \P(X)$ be two order reversing quasi involutions with the same image ${\mathcal C}$. Then $T\circ S: \P(X) \to {\mathcal C}$ is order preserving, and restricted to ${\mathcal C}$ it is a bijection. 	
\end{lem}

\begin{proof} Order preservation is immediate, the fact that the image is included in ${\mathcal C}$ is immediate, and bijectivity follows since if $A\in {\mathcal C}$ then $TS(STA) = TTA = A$. 
\end{proof}

Conjugating an \orqi $T$ by an order preserving bijection on ${\rm Im}(T)$, one gets an order reversing involution on the image. 

\begin{lem}
	Let $T: \P(X) \to \P(X)$ be an order reversing quasi involution, let $\C = T(\P(X))$ be its image, and let $R: {\mathcal C} \to {\mathcal C}$ be an order preserving bijection. Then $S = R^{-1}\circ T \circ R$ is an order reversing involution on ${\mathcal C}$. 
\end{lem}

\begin{proof}
	Let $A\subseteq B\in {\mathcal C}$. Then $RA\subseteq RB \in {\mathcal C}$ and so $TRA \supseteq TRB$, and using order preservation of $R^{-1}$ we get  
	$SA =  R^{-1}TRA \supseteq R^{-1}TRB = SB$.  The involution property follows from 
	$R^{-1}TRR^{-1}TRA = R^{-1}TTRA = R^{-1}RA = A$. 
\end{proof}

\begin{rem}
	Clearly ${\mathcal C}$ is closed under intersections as it is the image of an order reversing quasi involution.  Using Lemma \ref{lem:extending} we may thus extend $S$ to be an order reversing quasi involution
	on $\P(X)$. \end{rem}

Let us mention that the direction of the inclusion in the quasi involution property $(i)$ can be reversed, and there is a one-to-one correspondence between order reversing quasi involutions as defined in Definition \ref{def:set-duality}, and those that one could define with the reverse inclusion,  as is apparent from the following lemma. 
\begin{lem}\label{lemma:switch-direction}
	Let $T:\P(X)\to \P(X)$ be an order reversing quasi-involution. Then the mapping $S:\P(X) \to \P(X)$ given by $SK = X\setminus (T(X\setminus K))$ satisfies \begin{enumerate}[(i)]
		\item $K\supseteq SSK$, \tabto{5.5cm} (complemented quasi involution)
		\item if $L\subseteq K$ then $SK \subseteq SL$.  \tabto{5.5cm} (order reversion)
	\end{enumerate}
\end{lem}

\begin{proof}
	Indeed, for the order reversion property,  if $L\subseteq K$ then $X\setminus L \supseteq X\setminus K$ and so $T( X\setminus L)\subseteq T(X\setminus L)$, from which it follows that \[SL = X\setminus T( X\setminus L) \supseteq   X\setminus T( X\setminus K)  = SK.  \] 
	For the quasi involution property, we want to show $K\supseteq SSK$ which can be rewritten as $X\setminus K\subseteq X\setminus SSK$, which in turn means $X\setminus K\subseteq T(X\setminus SK) = TT(X\setminus K)$,
	which follows from the order reversion property given for $T$ on the subset $X\setminus K$. 	
\end{proof}

\begin{rem}
	Let us remark that if we change the quasi involution property $(i)$ in Definition \ref{def:set-duality} to be in the opposite direction, $TTK \subseteq K$, then by the above lemma the condition of being closed under intersection becomes a condition on being closed under unions, the condition of containing $X$ becomes the condition of containing $\emptyset$, and all of the claims we have proved have their obvious counterparts for this case. 
\end{rem}

As the reader will readily notice, the transform $S$ amounts to the composition of the given transform $T$ on both sides with the order reversing  involution given by $K\mapsto X\setminus K$.
Hence, we obtained a complemented \orqi $S$, namely an order reversing map satisfying $K\supseteq SSK$.  Whereas it may seem that in the specific choice of $RK = X\setminus K$ we only used that it is an order reversing involution, one quite easily checks that it is the {\em only} order reversing bijection on $\P(X)$.  However,  this construction can be generalized as follows.

\begin{lem}
	Let $T$ be an \orqi and $R$ be a complemented order reversing quasi involution. Then the composition $TRT$ is an \orqi while $RTR$ is a complemented order reversing quasi involution.
\end{lem}

\begin{proof}
	Assume that $T,R$ are order reversing and that for any $K\subseteq X$ we have $K\subseteq TTK$ and $RRK\subseteq K$.  It can be easily checked that a composition of an odd number of order reversing transforms remains order reversing. 
	To see that $TRT$ is a quasi involution note that
	\[TT(RTK) \supseteq RTK  \implies     RTTRTK \subseteq RRTK\subseteq TK   \implies    TRTTRTK\supseteq TTK \supseteq K.  \ \]
	Similarly, we get that the composition $RTR$ is a complemented quasi involution
	\[ RR (TR K) \subseteq TRK  \implies   
	TRRTRK \supseteq TT RK \supseteq RK  \implies   
	RTRRTRK \subseteq RRK \subseteq K.
	\]\end{proof}

\subsection{Restrictions}

If an \orqi is defined on a class $\C\subseteq\P(X)$ which is closed under intersections, we can define a new \orqi by restricting the original transform to some fixed subset. More precisely, we have the following proposition. 

\begin{prop}\label{prop:intersecting}
	Let $T:\C\to \C$ be an order reversing quasi involution on $\C$,  where $\C\subseteq \P(X)$ is closed under intersections.  Then for any fixed $M_0\in \C$, denoting $\C_{0}=\{K\cap M_0: K\in \C\}\subseteq \C$,  we have that $S:\C_{0}\to \C_{0}$ defined by $SK=TK\cap M_0$ is an order reversing quasi involution.
\end{prop}

\begin{proof}
	To see that $S$ is order reversing, consider $K,L \in \C_0$ such that $K\subseteq L$. Then since $T$ is order reversing we get that $TK\supseteq TL$.  Since intersecting with $M_0$ preserves that inclusion, we get that $S$ is indeed order reversing. 
	To see that for any $K\in \C_0$ we have $K\subseteq SSK$ note that $SK=TK\cap M_0\subseteq TK$. Since $\C$ is closed under intersections we may apply $T$ to both sides again and by the order reversing property of $T$ we get $T(TK\cap M_0) \supseteq TTK\supseteq K$.  Finally note that since $K\in \C_0$ it follows that $K\subseteq T(TK\cap M_0)\cap M_0=SSK$.
\end{proof}

\begin{rem}
	Instead of a direct proof one can use the cost associated with $T$ and restrict it to  $M_0\times M_0$. Such a cost exists by Lemma \ref{lem:extending} and Theorem \ref{thm:duality-induced-by-cost}.
\end{rem}

One can use Proposition \ref{prop:intersecting} to construct \orqis.   For example, considering the polarity transform on subsets of $\RR^n$, and letting $M_0 = RB_2^n$, a multiple of the Euclidean ball, the new transform, which is given by 
\[ SK = K^\circ \cap RB_2^n,\]
is an \orqi~ on subsets of $RB_2^n$. Its image, namely the class on which it is an involution, is the class of all convex subsets of $RB_2^n$ which include $\frac{1}{R}B_2^n$. If we instead chose $M_0$ to be a linear subspace, we would simply get polarity within the subspace. Taking $M_0$ to be some convex set gives the setting of Cordero-Erausquin's conjecture \cite{Cordero}.

Another natural way to restrict an \orqi on $X$ to a subset $Y$ it to look at a cost function on $X\times X$ inducing the \orqi and restrict it to $Y\times Y$. More precisely, given a space $X$ and a cost function $c:X\times X \to (-\infty , \infty] $ it generates a $c$-class $\mathcal{C}_X$ (image of the transform). Fix $Y\subseteq X$ and consider a cost function $\tilde{c}:Y\times Y \to (-\infty , \infty ]$ such that $\tilde{c}(x,y)=c(x,y)$ for all $x,y\in Y$. We want to study the $\tilde{c}$-class $\mathcal{C}_Y$ and it relation to $\mathcal{C}_X$.

\begin{lem}\label{lem:1}
	Let $X,Y,c, \tilde{c}$ as above. Assume that $Y\in \mathcal{C}_X$. It follows that 
	\[ \mathcal{C}_Y=\{ B\cap Y: \, Y^c \subseteq B \in \mathcal{C}_X\}.\]
\end{lem}

\begin{proof}%
	We will first show the inclusion ``$\subseteq$". 
	Let $A\in \mathcal{C}_Y$. This means that $A^{\tilde{c}\tilde{c}}=A\subseteq Y$, but
	\[ A^{\tilde{c}} = \cap_{x\in A} \{y\in Y: \, \tilde{c}(x,y)\ge 0\} = \cap_{x\in A} \{y\in Y: \, c(x,y)\ge 0\} = A^c\cap Y\]
	and further
	\[ A^{\tilde{c}\tilde{c}}= \cap_{z\in A^{\tilde{c}}} \{w\in Y: \, \tilde{c}(z,w)\ge 0\} =\cap_{z\in A^{\tilde{c}}\cap Y} \{w\in Y: \, c(z,w)\ge 0\} = (A^c \cap Y)^c\cap Y
	\]
	But the set $B=(A^c \cap Y)^c\in \mathcal{C}_X$. Moreover, if $y\in Y^c$ then $y\in (A^c \cap Y)^c$. Indeed, we have that $Y \supseteq A^c\cap Y$ and hence by order-reversion $Y^c \subseteq (A^c\cap Y)^c=B$.
	
	In the other direction, consider $B\in \mathcal{C}_X$ such that $Y^c\subseteq B$. We want to show that $B\cap Y \in \mathcal{C}_Y$. In fact, it turns out that $B\cap Y=(B^c)^{\tilde{c}}$ since $Y^c\subseteq B$ we have by order-reversion that $B^c \subseteq Y^{cc}=Y$, as $Y\in \mathcal{C}_X$ it is well defined. It follows that 
	\begin{align*}
	( B^c)^{\tilde{c}} &=\bigcap_{x\in B^c} \{y\in Y: \, \tilde{c}(x,y)\ge 0\} = \bigcap _{x\in B^c\cap Y} \{ y\in Y:\, c(x,y)\ge 0\} = (B^c\cap Y)^c\cap Y \\ &= B^{cc}\cap Y = B\cap Y.
	\end{align*}
	This completes the proof.
\end{proof}

We remark that in the case when $Y\notin \mathcal{C}_X$, i.e. when $Y^{cc}\neq Y$, it is not known if the statement still holds.

\begin{fact}
	$A\in \mathcal{C}_Y$ does not imply that $A\in \mathcal{C}_X$. Take for example unconditional bodies on $X=\R^n$ and $Y=(\R_+)^n$.
\end{fact}

\begin{fact}
	Given $\tilde{c}:Y\times Y \to (-\infty ,\infty]$ which generates $\mathcal{C}_Y$, there exists $X \supseteq Y$ and $c:X\times X \to (-\infty ,\infty]$ such that $\mathcal{C}_Y=\mathcal{C}_X$.
\end{fact}
\begin{proof}
	Define $c(x,y)=\tilde{c}(x,y)$ if $x,y\in Y$ and $-1$ otherwise. Then, for any $K\subseteq X$ we have
	\[ K^c = \cap _{x\in K} \{y \in X: \, c(x,y)\ge 0\} =\cap_{x\in K\cap Y} \{y \in Y: \, c(x,y)\ge 0\} = (K\cap Y)^{\tilde{c}}.\]
\end{proof}

The next fact is closely related to Lemma \ref{lem:1}, however for this direction one does not need to assume that $Y\in \mathcal{C}_X$.

\begin{fact}
	For any cost $c:X\times X\to (-\infty, \infty]$ and for any $Y\subseteq X$ we have that every set $A\in \mathcal{C}_Y$ can be written as $B\cap Y$, where $B\in \mathcal{C}_X$.
\end{fact}

\begin{proof}
	Let $A\in \mathcal{C}_Y$, i.e. $A=A^{\tilde{c}\tilde{c}}$. Then let $B=A^{cc}$ and we will show that $B\cap Y=A$.
	For the first inclusion, note that $A\subseteq A^{cc}$ and hence, since $A\subseteq Y$ and intersecting both sides with $Y$ we are done.
	
	In the other direction, let $y\in B\cap Y$. We will show that $y \in A\subseteq Y$.
	
	Assume towards a contradiction that $y\notin A$. Then, $y\notin A^{\tilde{c}\tilde{c}}$ and hence there exists some $a \in A^{\tilde{c}}$ such that $c(a,y)<0$. But, we have that $A^{\tilde{c}} \subseteq A^c$ as if $z\in A^{\tilde{c}}$ then for all $w\in A$ we have $\tilde{c}(z,w)\ge 0$ and $c|_Y=\tilde{c}$ hence $c(z,w)\ge0$ for all $w\in A$. Hence, we get the contradiction with $y\in A^{cc}\cap Y$ as we showed that there is some $a\in A^c$ with $c(a,y)<0$ and therefore $t\notin A^{cc}$.
\end{proof}

\bibliographystyle{plain}
\bibliography{ref_zoo}

\smallskip \noindent
School of Mathematical Sciences, Tel Aviv University, Tel Aviv 69978, Israel  
\smallskip \noindent

{\it e-mail}: shiri@tauex.tau.ac.il\\
{\it e-mail}: shaysadovsky@mail.tau.ac.il\\
{\it e-mail}: kasiawycz@outlook.com\\

\end{document}